\documentclass{amsart}

\usepackage[foot]{amsaddr}
\usepackage{amsthm}
\usepackage{amsmath}
\usepackage{graphicx}
\usepackage[margin=1.2in]{geometry}
\usepackage{amssymb,amsfonts}
\usepackage[all,arc]{xy}
\usepackage{enumerate}
\usepackage{mathrsfs}
\usepackage{bbm}

\DeclareMathOperator{\tr}{tr}
\DeclareMathOperator{\diam}{diam}

\numberwithin{equation}{section}

\newcommand{\rr}{\ensuremath{\mathbb{R}}}
\newcommand{\zz}{\ensuremath{\mathbb{Z}}}
\newcommand{\ep}{\ensuremath{\varepsilon}}
\newcommand{\bdry}{\ensuremath{\partial}}

\newcommand{\alphasemi}[2]{\ensuremath{[#1]_{C^{0,\alpha}(#2)}}}
\newcommand{\etasemi}[2]{\ensuremath{[#1]_{C^{0,\eta}(#2)}}}
\newcommand{\linfty}[2]{\ensuremath{||#1||_{L^{\infty}(#2)}}}
\newcommand{\conealpha}[2]{\ensuremath{||#1||_{C^{1,\alpha}(#2)}}}

\newcommand{\czeroone}[2]{\ensuremath{||#1||_{C^{0,1}(#2)}}}

\newcommand{\conegamma}[2]{\ensuremath{||#1||_{C^{1,\gamma}(#2)}}}

\newcommand{\etanorm}[2]{\ensuremath{||#1||_{C^{0,\eta}(#2)}}}

\newcommand{\tu}{\ensuremath{\tilde{u}}}

\newcommand{\bu}{\ensuremath{\bar{u}}}

\newtheorem{thm}{Theorem}[section]

\newtheorem{cor}[thm]{Corollary}
\newtheorem{prop}[thm]{Proposition}
\newtheorem{lem}[thm]{Lemma}

\theoremstyle{definition}

\newtheorem{defn}[thm]{Definition}

\theoremstyle{remark}
\newtheorem{rem}[thm]{Remark}

\begin{document}
\title[Error estimates]{Error estimates for approximations of nonhomogeneous nonlinear uniformly elliptic equations}
\author{Olga Turanova}
\address{Department of Mathematics, University of Chicago, 5734 S. University Avenue, Chicago, IL 60637}
\email{turanova@math.uchicago.edu}
\date{June 5, 2015}

\begin{abstract}
We obtain an error estimate between viscosity solutions and $\delta$-viscosity solutions of  nonhomogeneous fully nonlinear uniformly elliptic equations. The main  assumption, besides uniform ellipticity, is that the nonlinearity is  Lipschitz-continuous in space with linear growth in the Hessian. We also establish a rate of convergence for monotone and consistent finite difference approximation schemes for such equations.
\end{abstract}
\maketitle

\section{Introduction}
We prove an estimate between viscosity solutions and $\delta$-viscosity solutions of the boundary value problem
\begin{equation}\left\{
\begin{array}{l l}
\label{eqn for u}
F(D^2 u, x)=f(x) &\quad \text{ in }U \subset \rr^n,\\
u=g &\quad \text{ on }\bdry U,
\end{array}\right.
\end{equation}
where $F$ is uniformly elliptic (see (F\ref{ellipticity}) below) and Lipschitz-continuous in space with linear growth in the Hessian (see (F\ref{F lip}) below). As a consequence, we find a rate of convergence for monotone and consistent finite difference approximations to (\ref{eqn for u}). Both results generalize the work of Caffarelli and Souganidis in \cite{Approx schemes, Rates Homogen}, who consider either homogeneous equations or equations with separated dependence on the space variable and on the Hessian.

The nonlinearity $F$ is a continuous function on $\mathcal{S}_n \times U$, where $\mathcal{S}_n$ is the set of $n\times n$ real symmetric matrices endowed with the usual order and norm (for $X\in \mathcal{S}_n$,  $||X|| = \sup_{|v|=1}|Xv|$). We make the following assumptions:
\begin{enumerate}[({F}1)]
\item \label{ellipticity} $F$ is uniformly elliptic, which means there exist  constants $0<\lambda \leq \Lambda$ such that for all $x\in U$, any $X\in \mathcal{S}_n$, and for all $Y\geq 0$,
\[
\lambda||Y|| \leq F(X+Y, x)-F(X,x) \leq \Lambda||Y||;
\]
and,
\item
\label{F lip} 
there exists a positive constant $\kappa$ such that for all $x,y\in U$ and all $X\in \mathcal{S}_n$,
\[
|F(X,x)-F(X,y)|\leq \kappa|x-y|(||X||+1).
\]
\end{enumerate}

An example of  an equation satisfying our assumptions is the Isaacs equation 
\[
F(D^2u, x)=\sup_{\alpha}\inf_{\beta }L^{\alpha, \beta}u(x),
\] 
where, for each $\alpha$ and $\beta$ in some index sets, the operator $L^{\alpha,\beta}$ is given by
\[
L^{\alpha, \beta}u(x)= \sum_{i,j=1}^n a_{ij}^{\alpha, \beta}(x)\partial^2_{i,j}u(x)+f^{\alpha,\beta}(x),
\]
and is uniformly elliptic with uniformly Lipschitz coefficients, which means there exists a $\kappa$ such that for all $x,y\in U$ and for all $\alpha$, $\beta$, $i$ and $j$,
\[
|a_{ij}^{\alpha, \beta}(x)-a_{ij}^{\alpha, \beta}(y)|\leq \kappa|x-y|
\] 
and 
\[
|f^{\alpha,\beta}(x)-f^{\alpha,\beta}(y)|\leq \kappa |x-y|.
\]
The Isaacs equation arises in the study of stochastic differential games. We do not give further details about the Isaacs equation and refer the reader to Section 1 of Crandall, Ishii and Lions' \cite{User's guide} for a list of references.

We also assume: 
\begin{enumerate}[({U}1)]
\item
\label{assum U}
$U$ is a bounded subset of $\rr^n$ with regular boundary,
\end{enumerate}
\begin{enumerate}[({G}1)]
\item 
\label{f assumption}$f\in C^{0,1}(U)$, and
\item
\label{g c one alpha}
$g\in C^{1,\gamma}(\bdry U)$ for some $\gamma \in (0,1]$.
\end{enumerate}

The main result is a comparison between solutions and $\delta$-viscosity solutions (briefly, $\delta$-solutions) of (\ref{eqn for u}). The definition of $\delta$-solutions is given in Section \ref{solution}. Next we present a statement of our main result that has been simplified for the introduction; the full statement is in Section \ref{sec pf of main thm}.

\begin{thm}
\label{vague main thm}
Assume  (U\ref{assum U}), (F\ref{ellipticity}), (F\ref{F lip}), (G\ref{f assumption}), (G\ref{g c one alpha}). Let $u$ be a  viscosity solution of (\ref{eqn for u}) and assume that $\left\{v_\delta\right\}_{\delta\geq 0}$ is a family of H\"older continuous $\delta$-solutions of (\ref{eqn for u}) that satisfy, for all $\delta>0$,
 \[
 v_\delta=u \text{ on }\bdry U.
 \]
There exist positive constants $\bar{\delta}$, $\bar{\alpha}$ and $\bar{c}$ such that,  for any $\delta\leq \bar{\delta}$,
\[
\sup_U |u-v_\delta|\leq \bar{c}\delta^{\bar{\alpha}}.
\]
\end{thm}

The notion of $\delta$-solutions was introduced in \cite{Rates Homogen} for fully nonlinear uniformly elliptic equations of the form $F(D^2u)=0$. An error estimate between solutions and $\delta$-solutions of such equations was established in \cite{Rates Homogen}.  A key step in the proof of the error estimate was a regularity result (\cite[Theorem A]{Approx schemes}), which says that, outside of sets of small measure, solutions of $F(D^2u)=f(x)$ have second-order expansions with controlled error.
The proof of this regularity theorem relies on the equation being homogeneous -- differentiating  $F(D^2u)=f(x)$ implies that the derivatives $u_{x_i}$  solve the linear uniformly elliptic equation $\tr (DF \cdot D^2u_{x_i})=f_{x_i}$; therefore, a known estimate that gives first-order expansions on large sets (see Chapter 7 of Caffarelli and Cabre \cite{Cabre Caffarelli book}) applies to $u_{x_i}$, and from this the estimates on $u$ are deduced.  

The main challenge in the $x$-dependent case is that  this extra regularity result is not known. Because differentiating the $x$-dependent equation $F(D^2u,x)=f(x)$ does not imply anything useful about the derivatives of $u$, we cannot hope to replicate the proof of \cite[Theorem A]{Approx schemes} for non-homogeneous equations. 

Let us briefly describe the strategy of proof of Theorem \ref{vague main thm}. We perturb the nonlinearity $F$ and then ``localize" the equation (see Propositions \ref{compare u and uep} and \ref{compare to constant coef}). This allows us obtain an approximation of $u$ that is regular enough to compare to the $\delta$-solution $v_\delta$.  We hope that this method, in particular the use of the perturbations of $F$, may be of interest in other contexts. We include a detailed outline in Section \ref{subsec idea}. 

\begin{rem}
\label{remark}
The assumption (F\ref{F lip}) on the nonlinearity $F$ may be weakened. In fact, there exists a universal constant $\alpha$ such that if $\beta >1-\alpha$, our results hold for any $F$ that satisfies
\[
|F(X,x)-F(X,y)|\leq \kappa(1+||X||)|x-y|^\beta.
\]
This requirement on $\beta$ comes from the proof of Proposition \ref{compare to constant coef}, and $\alpha$ is the exponent from Proposition \ref{c one alpha prop}. For simplicity, we will only work with the case $\beta = 1$ (in other words, we assume that $F$ satisfies (F\ref{F lip})).
\end{rem}
\begin{rem}
We often assume
\begin{equation}
\label{last prop of F}
F(0,x) =0\text{ for all }x\in U.
\end{equation}
This is not a restrictive assumption:  the equation $F(D^2u,x)=f(x)$ is equivalent to $F(D^2u , x)-F(0,x)=f(x)-F(0,x)$, and the nonlinearity $\tilde{F}(M,x)=F(M,x)-F(0,x)$ satisfies (\ref{last prop of F}). 
\end{rem}

We also study finite difference approximations to (\ref{eqn for u}).  We write the finite difference approximations as 
\begin{equation}\left\{
\begin{array}{l l}
\label{eqn for uhh}
F_h[v_h](x)=0 &\quad \text{ in }U_h, \\
v_h=g &\quad \text{ on }\bdry U_h,
\end{array}\right.
\end{equation}
where $U_h = U\cap h\zz^n$ is the mesh of discretization and $F_h$ is the finite difference operator. We assume:
\begin{enumerate}[({$F_h$}1)] 
\item \label{monotone} if $v_h^1$ and $v_h^2$ satisfy $F_h[v_h](x)=0$ in $U_h$ and $v_h^1\leq v_h^2$ on $\bdry U_h$, then $v_h^1\leq v_h^2$ on $U_h$; and
\item
\label{consistent} 
 there exists a positive constant $K$ such that for all $\phi\in C^3(U)$ ,
\[
|F_h[\phi](x) - (F(D^2 \phi, x)-f(x))|\leq K(1+\linfty{D^3\phi}{U})h \text{ in }U.
\]
\end{enumerate} 
Schemes that satisfy ($F_h$\ref{monotone}) and ($F_h$\ref{consistent}) are said to be, respectively, \emph{monotone} and \emph{consistent with an error estimate for $F$}.

We have simplified our notation here in order to state our main result; all the details about approximation schemes and the precise statement of Theorem \ref{vague main thm h}  are given in Sections \ref{sec h} and \ref{sec proof of theorem h}. We prove:
\begin{thm}
\label{vague main thm h}
Assume (U\ref{assum U}), (F\ref{ellipticity}), (F\ref{F lip}), (G\ref{f assumption}), and suppose $g\in C^{1,\gamma}(U)$. 
Assume that $F_h$ is a monotone scheme that is consistent with an error estimate for $F$. Assume that $u$ is the viscosity solution of (\ref{eqn for u}) and that $v_h$ satisfies (\ref{eqn for uhh}). There exist positive constants $\bar{c}$, $\bar{\alpha}$ and $\bar{h}$ such that for all $h\leq \bar{h}$,
\[
\sup_{U_h} |u-v_h|\leq \bar{c}h^{\bar{\alpha}}.
\]
\end{thm}

The convergence of monotone and consistent approximations of fully nonlinear second order PDE was first established by Barles and Souganidis \cite{BarlesSouganidis}. Kuo and Trudinger \cite{KT estimates, KT discrete methods} later studied the existence of monotone and consistent approximations for nonlinear equations and the regularity of the approximate solutions $u_h$. They showed that if $F$ is uniformly elliptic, then there exists a monotone finite difference scheme $F_h$ that is consistent with $F$, and that the approximate solutions $v_h$ are in $C^{0,\eta}$. However, obtaining an error estimate remained an open problem. 

The first error estimates for approximation schemes were established by Krylov for equations that are either convex or concave, but possibly degenerate  \cite{Krylov 1998, Krylov 1999}. Krylov used stochastic control methods that apply in the convex  or concave case, but not in the general setting.  Barles and Jakobsen in \cite{BarlesJakobsen2002, BarlesJakobsen2005} improved Krylov's error estimates for  convex or concave equations. In \cite{Krylov2005} Krylov improved the error estimate to be of order $h^{1/2}$, but still in the convex/concave case. In addition,  Jakobsen \cite{Jakobsen2004, Jakobsen2006} and Bonnans, Maroso, and Zidani \cite{Bonnans} established error estimates for special equations or for special dimensions. The first error estimate for general nonlinear equations that are neither convex nor concave was obtained by Caffarelli and Souganidis in \cite{Approx schemes}.  Their result holds for equations $F$ that do not depend on $x$.

To our knowledge, Theorem \ref{vague main thm h} is  the first  error estimate for general nonlinear uniformly elliptic equations that are neither convex nor concave and are not homogeneous. In particular, this is the first error estimate for approximations of the Isaacs equation. To prove Theorem \ref{vague main thm h}, we show that an appropriate regularization of the solution of (\ref{eqn for uhh}) is a $\delta$-solution of (\ref{eqn for u}), where $\delta$ depends on $h$ (see Proposition \ref{v_h is delta soln}). This allows us to essentially deduce Theorem \ref{vague main thm h} from our estimate in Theorem \ref{vague main thm}.

Our paper is structured as follows. In Section \ref{prelim} we establish notation, give the definition of $\delta$-solutions and state several known results about the regularity of viscosity solutions of (\ref{eqn for u}). We provide a detailed outline of the proofs of our main results in Section \ref{subsec idea}.   Section \ref{comparison sec}  is devoted to establishing an estimate between the solution $u$ of (\ref{eqn for u}) and solutions of the equation with ``frozen coefficients" on small balls. This is Proposition \ref{compare to constant coef}. In Section \ref{section ep} we study  perturbations of the equation (\ref{eqn for u}) and prove an estimate between  $u$ and solutions of the perturbed equations (Proposition \ref{compare u and uep}). In Section \ref{elem lemm sec} we establish an elementary lemma that plays an important role in the proofs of Theorems \ref{vague main thm} and \ref{vague main thm h}. Section \ref{main lem sec} is devoted to the statement and proof of  an important estimate between $\delta$-solutions of (\ref{eqn for u}) and solutions of equations with frozen coefficients. This is Proposition \ref{main lem}.  The full statement and the proof of our main result, Theorem \ref{vague main thm}, is in Section \ref{sec pf of main thm}. Section \ref{sec h} is devoted to introducing the necessary notation and stating known results about approximation schemes. In Proposition \ref{v_h is delta soln} we show that certain regularizations of the approximate solutions $v_h$ are $\delta$-solutions of (\ref{eqn for u}). In Section \ref{sec proof of theorem h} we give the precise statement and proof of Theorem \ref{vague main thm h}.  In  Appendix  A we state several known results related to the comparison principle for viscosity solutions. In Appendix B we summarize the properties of inf- and sup- convolutions that we use in our paper.

\tableofcontents

\section{Preliminaries and outline}
\label{prelim}
In this section we establish notation, give the definition of $\delta$-solutions, recall some known results, and provide an outline of our argument. 

\subsection{Notation}
\label{notation}
We denote open balls in $\rr^n$ by 
\[
B_r(x_0)=B(x_0,r)=\{x\in \rr^n: \   \   |x-x_0|<r\},
\]
and we often write $B_r$ to mean $B_r(0)$. We denote the diameter of $U\subset \rr^n$ by $\diam U$. A \emph{paraboloid} $P(x)$ is a polynomial in $x_1,...,x_n$ of degree 2. We say that a paraboloid $P$ is of \emph{opening} $M$ if 
\[
P(x)=l(x)+ M\frac{|x|^2}{2},
\]
where $l$ is an affine function and $M$ is a constant.

Throughout the paper we say a constant is \emph{universal} if it is positive and depends only on $n$, $\lambda$, and $\Lambda$. 

\subsection{Notions of solution}
\label{solution}
We consider solutions of (\ref{eqn for u}) in the viscosity sense; see 
\cite{User's guide} for an introduction to the theory of viscosity solutions. 
Throughout, we say ``solution" to mean ``viscosity solution."
\begin{defn}
We say that $v$ is a \emph{$\delta$-subsolution} (respectively, $\delta$-\emph{supersolution}) of (\ref{eqn for u}) if, for any $x$ such that $B_{\delta}(x)\subset U$, any paraboloid $P$ with $P(x)=v(x)$ and $P(y)\geq v(y)$  (respectively, $P(y)\leq v(y)$) for all $y$ in $B_{\delta}(x)$ satisfies
\[
F(D^2P, x)\geq 0 \text{ (respectively, $F(D^2P, x)\leq 0 $)}.
\]
We say that $v$ is a $\delta$-\emph{solution} if $v$ is both a $\delta$-subsolution and a $\delta$-supersolution.
\end{defn}
From the definition, it is clear that a viscosity solution of (\ref{eqn for u}) is a $\delta$-solution of (\ref{eqn for u}) for any $\delta>0$. The difference from the definition of viscosity solution is that for $u\in C(U)$ to be a $\delta$-supersolution (resp. $\delta$-subsolution), any test paraboloid must stay below (resp. above) $u$ on a set of fixed size.

\subsection{Known results}
\label{know}
 We recall that the \emph{concave envelope} of a function  $u\in C(B_r)$ is defined as 
\[
\Gamma_u(x)=\inf\{ l(x): \  l\geq u \text{ in $B_r$ and $l$ is affine}\}.
\]
In addition, we will use the following terminology:
\begin{defn}
For $u\in C(U)$, we say $D^2u(x)\geq MI$ (resp. $D^2u\leq MI$) \emph{in the sense of distributions} if there exists a paraboloid $P$ of opening $M$ such that $u(x)=P(x)$ and, for all $y\in B_r(x)$ for some $r$, 
\[
u(y)\geq P(y) \text{ (resp. }u(y)\leq P(y)).
\]
\end{defn}
The following fact is a key step in the proof of the well-known Alexander-Bakelman-Pucci (ABP) estimate, and it will play a central role in our arguments. It is Lemma 3.5 of the book of Caffarelli and Cabre \cite{Cabre Caffarelli book}, modified slightly for our setting. 
\begin{prop}
Assume $u\in C(B_r)$ is such that $u\leq 0$ on $\bdry B_r$.  
Assume that there exists a constant $K$ such that $D^2 u(x)\geq K$ in the sense of distributions for all $x\in B_r$. There exists a universal constant $C$ such that 
\[
\sup_{B_r} u \leq Cr\left(\int_{\{u=\Gamma_u\}}|\det D^2\Gamma_u|\right)^{\frac{1}{n}},
\]
where $\Gamma_u$ is the concave envelope of $u^+$ in $B_{2r}$. Moreover, $\Gamma_u$ is twice differentiable almost everywhere.
\label{pre ABP}
\end{prop}

Next we provide the statements of several known regularity results, for which we introduce the adimensional $C^{1,\alpha}$ norm, denoted by $\conealpha{\cdot}{\bar{B}_r}^*$:
\[
\conealpha{u}{\bar{B}_r}^*=\linfty{u}{B_r}+r\linfty{Du}{B_r}+r^{1+\alpha}\alphasemi{Du}{B_r}.
\]
We  need the following rescaled version of the interior $C^{1,\alpha}$ estimate  \cite[Corollary 5.7]{Cabre Caffarelli book}. 
\begin{prop}[Interior estimate]
\label{interior}
Assume (F\ref{ellipticity}). There exist universal constants $\alpha$ and $C$ such that if $u$ is a viscosity solution of $F(D^2u)=0$ in $B_r$, then $u\in C^{1,\alpha}(\bar{B}_{r/2})$ and
\[
\conealpha{u}{\bar{B}_{r/2}}^* \leq C(\linfty{u}{B_r}+|F(0)|).
\]
\end{prop}

In addition to the interior $C^{1,\alpha}$ interior estimate, we need the following global $C^{1,\alpha}$  estimate (Winter \cite[Theorem 3.1,  Proposition 4.1]{Winter}).

\begin{prop}[Global estimate]
\label{c one alpha prop}
Assume (U\ref{assum U}), (F\ref{ellipticity}), (F\ref{F lip}), (G\ref{f assumption}), (G\ref{g c one alpha}) and $F(0,x)\equiv 0$. There exists a universal constant $\alpha$  and a positive constant $C$ that depends on $n$, $\lambda$, $\Lambda$, $\kappa$, $\diam U$ and the regularity of $\bdry{U}$, such that  if $u$ is the viscosity solution of
\begin{equation*}\left\{
\begin{array}{l l}
F(D^2 u, x)=f(x) &\quad \text{ in }U,\\
u=g &\quad \text{ on }\bdry U,
\end{array}\right.
\end{equation*}
then $u\in C^{1,\alpha}(U)$ and
\[
\conealpha{u}{U} \leq C(\linfty{f}{U}+\conegamma{g}{\bdry U}).
\]
\end{prop}

Finally, we state the following lemma that we'll employ throughout the argument. Its proof is an elementary barrier argument, which we include in Appendix \ref{appendix Lemma} for the sake of completeness.
\begin{lem}
\label{barrier}
Assume (F\ref{ellipticity}) and  (U1), and let $c$ be a positive constant. Assume that $u$ is a solution of (\ref{eqn for u}) and $\bar{u}$ is a solution of 
\begin{equation*}
\left\{
\begin{array}{l l}
F(D^2 \bar{u}, x)=f(x)+c &\quad \text{ in } U,\\
u=g &\quad \text{ on }\bdry U.
\end{array}\right.
\end{equation*}
Then for all $x\in U$ we have
\[
\bar{u}(x)\leq u(x) \leq \bar{u}(x) +\frac{c}{2\lambda}(\diam U)^2.
\]
\end{lem}

\subsubsection{Regularization by inf- and sup- convolution}
\label{infsup}
We will use the technique of regularization by inf- and sup- convolution. This is an important tool in the regularity theory for viscosity solutions. We refer the reader to \cite[Section 5]{Cabre Caffarelli book} and  \cite[Section 5]{Rates Homogen}  for a thorough introduction. For our purposes, we recall the relevant definitions:
\begin{defn}
\label{defn conv}
For $u\in C(U)$ and for $\theta>0$, we define   the \emph{sup-convolution} $u^{\theta, +}$ and \emph{inf-convolution} $u^{\theta, -}$  as
\[
u^{\theta, +}(x) = \sup_{y\in U}\left\{u(y)-\frac{|x-y|^2}{2\theta}\right\}, \  u^{\theta, -}(x) = \inf_{y\in U}\left\{u(y)+\frac{|x-y|^2}{2\theta}\right\}.
\]
\end{defn}
\begin{defn}
\label{defn conv set}
Given $\theta>0 $, $\delta>0$ and $u\in C(U)$, we define the subset $U^{\theta}_\delta$ of $U$ by set, 
\[
U^{\theta}_\delta = \{ x\in U \  | \  d(x, \bdry U)> 2 \theta^{1/2} \linfty{u}{U}^{1/2} +\delta\} .
\]
\end{defn}
We summarize the basic properties of inf- and sup- convolutions in Proposition \ref{prop conv} of the appendix. One property is of particular importance  --  taking inf- and sup- convolution preserves the notion of super- and sub- solution, as well as of $\delta$-super and $\delta$-sub solution. We state this precisely in items (\ref{item solves}) and (\ref{item delta solves}) of Proposition \ref{prop conv}.

Theorem A of  \cite{Approx schemes} says that if $w$ satisfies a homogeneous equation on a ball $B_\rho$, then $w$ has second order expansions with controlled error on large parts of a smaller ball. The analogue of this result for inf- and sup- convolutions is Proposition 1.2 of \cite{Approx schemes}, which we state as Proposition \ref{thm A}. This result plays a key role in our argument.

\begin{prop}
\label{thm A}
Assume  (F\ref{ellipticity}), $f\in C^{0,1}(U)$ and fix some $x_0$ and $\hat{x}$ in $U$. Let $w\in C^{0,1}(B_{\rho}(\hat{x}) )$ be a viscosity solution of 
\[
F(D^2w, x_0)=f(x) \text{ in } B_{\rho}(\hat{x}).
\]
We denote $B^\theta_{\rho}(\hat{x}) = B(\hat{x}, \rho-2\theta\linfty{Dw}{B_\rho(\hat{x})})$.

There exist universal constants $\sigma$, $t_0$ and $C$ such that for any $t>t_0$  there exists an open set $A^+_t\subset B^{\theta}_{\rho}(\hat{x})$  (respectively, $A^-_t\subset B^{\theta}_{\rho}(\hat{x})$) such that
\[
|B^{\theta}_{\rho/2}(\hat{x}) \setminus A^{\pm}_t| \leq C\rho^{n-\sigma}(\linfty{Du}{B_\rho(\hat{x})}^\sigma+\linfty{Df}{B_\rho(\hat{x})}^\sigma) t^{-\sigma},
\]
and for all $x\in  A_t^{\pm}\cap B_{\rho/2}(\hat{x})$ there exists a quadratic polynomial $P$ with
\[
F(D^2P, x_0)=0,
\]
and, for all $y\in B^{\theta}_\rho(\hat{x})$,
\begin{align*}
&w^{\theta,+}(y)\geq w^{\theta, +}(x) +P(y) -Cr^{-1}t|y-x|^3\\
&(\text{respectively, }w^{\theta,-}(y)\leq w^{\theta, -}(x) +P(y) +Cr^{-1}t|y-x|^3).
\end{align*}
\end{prop}

\subsection{Outline of the proof of the main results}
\label{subsec idea}
 We outline the proof of the upper bound on $u-v_\delta$ that is asserted by Theorem \ref{vague main thm}; the proof of the lower bound is similar. 
Let us use $m$ to denote $\sup (u-v_\delta)$.  We may assume $m>0$, as otherwise there is nothing to prove. We use an elementary fact about H\"older-continuous functions,  Lemma \ref{elem lem}, to find a point $x_0\in U$ at which  $u-v_\delta$  is touched from above by a concave paraboloid of opening $Cm$. In order to present our ideas most clearly, let us assume for the purposes of the outline that the paraboloid is exactly $(u-v_\delta)(x_0) - Cm|x-x_0|^2$. (See Figure 1.)  Since this paraboloid touches $u-v_\delta$ from above at $x_0$, we find,  for any positive $r$,
\begin{equation*}
\sup_{\bdry B_r(x_0)} (u-v_\delta)\leq (u-v_\delta)(x_0)-Cmr^2.
\end{equation*}
Rearranging the previous line gives a bound on $m$ in terms of how much $u-v_\delta$ changes on the small ball $B_r(x_0)$:
\begin{equation}
\label{outlineone}
Cmr^2 \leq  (u-v_\delta)(x_0)- \sup_{\bdry B_r(x_0)} (u-v_\delta).
\end{equation}

\begin{figure}\caption{Illustration for Subsection \ref{subsec idea}.} \centering \includegraphics[trim = 0cm 0cm 0cm 0cm  , clip, scale=.3]{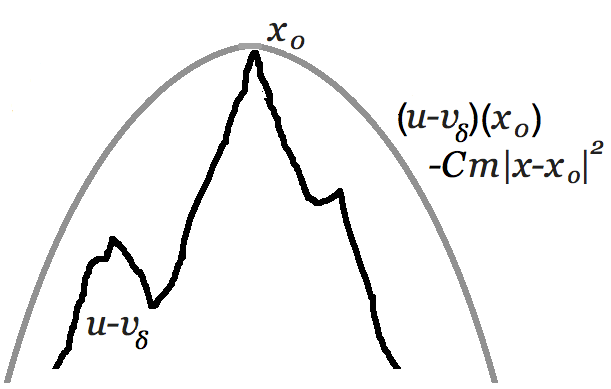}\end{figure}  

Next, we consider the solution of the equation on $B_r(x_0)$, but with ``frozen" coefficients: let $\tu$ be the solution of
\begin{equation*}\left\{
\begin{array}{l l}
F(D^2 \tu, x_0)=f(x_0) &\quad \text{ in }B_r(x_0),\\
\tu=u &\quad \text{ on }\bdry B_r(x_0).
\end{array}\right.
\end{equation*}
 Proposition \ref{compare to constant coef} says $\linfty{u-\tu}{B_r(x_0)}\leq Cr^{2+\alpha}$. Using this to estimate the right-hand side of   (\ref{outlineone}) from above yields,
\begin{equation}
\label{outline2}
Cmr^2 \leq  (\tu-v_\delta)(x_0) -\sup_{\bdry B_r(x_0)}  (\tu -v_\delta ) +Cr^{2+\alpha}.
\end{equation}
Since $\tu$ is the solution to a homogeneous equation, the regularity result Proposition \ref{thm A},  implies that $\tu$ has second order expansions with controlled error on large portions of $B_r(x_0)$. As in \cite{Approx schemes}, this extra regularity of $\tu$ allows us  to compare $\tu$ and $v_\delta$ on $B_r(x_0)$ and conclude  
\[
\sup_{B_r(x_0)}( \tu - v_\delta)  -\sup_{\bdry B_r(x_0)}( \tu - v_\delta)\leq r^2\delta^\alpha.
\]
(This argument is   Proposition \ref{main lem}). 
Using the previous estimate to bound the right-hand side of (\ref{outline2}) from above yields,
\[
Cmr^2\leq C\delta^\alpha r^2 + Cr^{2+\alpha}.
\]
Dividing by $r^2$ and choosing $r\leq \delta$ gives the desired estimate on $m$.

The main difference between the outline and the actual proof is that we need to work with inf- and sup- convolutions $v_\delta^{\theta, -}$ and $u^{\theta, +}$. Because we need take inf- and sup- convolutions inside of the small ball $B_r(x_0)$,  the radius $r$ has to be bigger than the parameter $\theta$ of the inf- and sup- convolutions. This restriction leads to problems. To get around them, we introduce the perturbations $F_\ep$ and $F^\ep$ of the equation itself. We define $F_\ep$ and $F^\ep$ by,
\[
F_\ep(X,x)=\inf_{y\in B_\ep (x)\cap U}F(X,y),\    \text{ and } F^\ep (X,x) =\sup_{y\in B_\ep (x) \cap U}F(X,y).
\]
We similarly define $f_\ep$ and $f^\ep$; please see Definition \ref{def ep} for the details. We let $u_\ep$ be the solution of,
\begin{equation*}\left\{
\begin{array}{l l}
F_\ep(D^2 u_\ep, x)=f^\ep(x) &\quad \text{ in }U,\\
u_\ep=g &\quad \text{ on }\bdry U.
\end{array}\right.
\end{equation*}
Proposition \ref{compare u and uep} asserts 
\[
\linfty{u_\ep -u}{U}\leq C\ep.
\]
We take $\ep \leq \delta^\alpha$ and replace $u$ with $u_\ep$. According to the previous line, that the error we make is of size $C\ep\leq C\delta^\alpha$, which does not affect the final estimate. Next, we ``freeze the coefficients" of $F_\ep$: we consider the solution $\tu_\ep$ of 
\begin{equation}\left\{
\begin{array}{l l}
F_\ep(D^2 \tu_\ep, x_0)=f^\ep(x_0) &\quad \text{ in }B_r(x_0),\\
\bar{u}_\ep=u_\ep &\quad \text{ on }\bdry B_r(x_0).
\end{array}\right.
\end{equation}
Then we proceed as explained in the first part of the outline. We do this for $\ep \gg r$, so that the equation with frozen coefficients ``sees" outside of $B_r(x_0)$. This detail is extremely important and allows us to regularize using the inf- and sup- convolutions and complete the argument.

The proof of Theorem \ref{vague main thm h}  follows essentially the same outline as the proof of Theorem \ref{vague main thm}. This because, as we show  in Proposition \ref{v_h is delta soln}, certain regularizations of the approximate solutions $v_h$ are $\delta$-solutions of (\ref{eqn for u}).

\section{The estimate between the solution of (\ref{eqn for u}) and solutions of (\ref{eqn for u}) with fixed coefficients.  }
\label{comparison sec}
In this section we establish the following result, which allows us to estimate the difference between $u$ and solutions of the equation with ``frozen" coefficients:
\begin{prop}
\label{compare to constant coef}
Assume (U\ref{assum U}), (F\ref{ellipticity}), (F\ref{F lip}), (G\ref{f assumption}), (G\ref{g c one alpha}) and $F(0,x)\equiv 0$. Let $u$ be the viscosity solution of (\ref{eqn for u}). There exist a universal constant $\alpha$,  a positive constant $r_0= r_0(\lambda, \Lambda, n, \kappa)$ and a positive constant $C$ that depends on $\lambda, \Lambda$, $n$, $\kappa$, $\czeroone{f}{U}$, $\conegamma{g}{\bdry U}$, $\diam U$ and the regularity of $\bdry U$ such that if,  we fix $r$ with $0<r<r_0$, a point $x_0$ such that $B_r(x_0)\subset U$, and take $\tu$ to be the viscosity solution of 
\begin{equation*}
\left\{\begin{array}{l l}
F(D^2 \tu, x_0)=f(x_0) &\quad \text{ in }B_r(x_0),\\
\tu=u &\quad \text{ on }\bdry B_r(x_0),
\end{array}\right.
\end{equation*}
then
\[
\linfty{ u-\tu }{B_r(x_0)} \leq Cr^{2+\alpha}.
\]
\end{prop}
 The proofs of Proposition \ref{compare to constant coef} and Proposition \ref{compare u and uep}, which we present in the next section, are similar to the proof of the comparison principle for uniformly elliptic equations of  Ishii and  Lions' \cite[Theorem III.1]{Ishii Lions}. 
At the heart of the proof of both propositions is the following lemma, which combines Theorem 3.2 of \cite{User's guide}  and Lemma III.1 of \cite{Ishii Lions}.  For the convenience of the reader, we give the statements of Theorem 3.2 of \cite{User's guide}  and Lemma III.1 of \cite{Ishii Lions} in Appendix A.

\begin{lem}
\label{lem compare}
There exists a constant C(n) that depends only on $n$ such that the following holds. Assume $V$ is an open subset of $\rr^n$ and $u,v\in C(V)$ are viscosity solutions of $F(D^2u, x)=f(x)$ and $G(D^2v, x)=g(x)$ in $V$. Suppose that $(x_a, y_a)\in V\times V$ is a local maximum of 
\begin{equation}
\label{double var lem}
u(x) - v(y) - \frac{a}{2}|x-y|^2.
\end{equation}
In addition, assume that there exist $s,t\in \rr$ with 
\begin{equation}
\label{assume t}
t\leq \frac{\lambda}{2C(n)},
\end{equation}
and such that 
\begin{equation}
\label{assume eqn}
\text{ for any $M, N\in \mathcal{S}_n$ with $M\leq N$, we have }
F(M, x_a)-G(N,y_a)\leq t||M|| +s-\lambda ||N-M||.
\end{equation}
Then, 
\begin{equation}
\label{conclusion lem compare}
f(x_a)-g(y_a)-s\leq \frac{t^2C^2(n)a}{2\lambda}.
\end{equation}
\end{lem}

\begin{proof}
We take $C(n)$ to be the constant from Lemma \ref{matrix lemma}. 
Since $(x_a,y_a)$ is an interior maximum of the quantity (\ref{double var lem}), Theorem \ref{user guide thm} implies that there exist $X, Y \in\mathcal{S}_n$ such that
\begin{equation}
\label{matrix ineq use1}
 -3a\left( \begin{array}{cc}
I & 0  \\
0 & I  \\
\end{array}
 \right) \leq
\left( \begin{array}{cc}
X & 0  \\
0 & -Y  \\
\end{array}
 \right)
 \leq
 3a\left( \begin{array}{cc}
I & -I  \\
-I & I  \\
\end{array}
 \right),
\end{equation}
\begin{equation}
\label{F at xa}
F(X, x_a)\geq f(x_a),
\end{equation}
and
\begin{equation}
\label{F at ya}
G(Y, y_a) \leq g(y_a).
\end{equation}
Subtracting (\ref{F at ya}) from (\ref{F at xa}), we find 
\begin{equation}
\label{fgFG}
f(x_a)-g(y_a)\leq F(X,x_a)-G(Y,y_a).
\end{equation}
We point out that the matrix inequality (\ref{matrix ineq use1}) implies  $X\leq Y$.
By assumption (\ref{assume eqn}), we therefore have
\[
F(X,x_a)-G(Y,y_a)\leq  t||X||+s-\lambda ||Y-X||.
\]
We use (\ref{fgFG}) to bound the left-hand side of the previous  line from below and find,
\begin{equation}
\label{fgxa}
f(x_a)-g(x_a) \leq t||X||+s-\lambda ||Y-X||.
\end{equation}
 In addition, since $X$ and $Y$ satisfy (\ref{matrix ineq use1}), Lemma \ref{matrix lemma} implies,
\begin{equation}
\label{what matrix lemma implies}
||X||\leq C(n)\left\{a^{1/2}||X-Y||^{1/2}+||X-Y||\right\}.
\end{equation}
We use (\ref{what matrix lemma implies}) to bound the first term on the right-hand side of (\ref{fgxa}) from above and obtain,
\[
f(x_a)-g(x_a)\leq tC(n)\left\{a^{1/2}||X-Y||^{1/2}+||X-Y||\right\} +s-\lambda||X-Y||.
\]
Rearranging yields,
\[
f(x_a)-g(x_a)-s\leq ||X-Y|| ( tC(n) -\lambda) +tC(n)a^{1/2}||X-Y||^{1/2}.
\]
According to the upper bound (\ref{assume t}) on $t$, we have $tC(n) -\lambda\leq -\lambda/2$. We use this to estimate the first term on the right-hand side of the previous line from above, and find,
\[
f(x_a)-g(x_a)-s\leq  -\frac{\lambda}{2}||X-Y||+tC(n)a^{1/2}||X-Y||^{1/2}
,
\]
Since the right-hand side is a quadratic polynomial in $z=||X-Y||^{1/2}$ with negative leading coefficient, we obtain
\[
f(x_a)-g(x_a)-s\leq \sup_z \left\{-\frac{\lambda}{2}z^2 +tC(n)a^{1/2}z\right\} = \frac{t^2C^2(n)a}{2\lambda},
\]
as desired.
\end{proof}

For the proof of Proposition \ref{compare to constant coef}, we first rescale to a ball of radius 1, where we double variables and then apply Lemma \ref{lem compare}. We will  need to keep careful track of all the parameters once we double variables.  In addition, in order to  apply Lemma \ref{lem compare}, we will need to verify that the point $(x_a,y_a)$  at which the supremum in  (\ref{double var lem}) is achieved is contained in the interior of $V\times V$. For this, we need the following lemma. Its proof is elementary and is provided in Appendix A.

\begin{lem}
\label{doubling variables property}
Suppose $v, w \in C^{0,1}(V)$ with $v=w$ on $\bdry V$. Then, for all $a>0$,
\[
\sup_{(\bdry V \times V) \cup (V\times \bdry V)} \left( v(x)-w(y) - \frac{a}{2}|x-y|^2\right) \leq 2(\linfty{Dv}{V}^2+\linfty{Dw}{V}^2)a^{-1};
\]
and, if $(x_a,y_a)\in V\times V$ is a point at which the supremum on the left-hand side of the previous line is achieved, then
\[
|x_a-y_a|\leq 2a^{-1}\min\{\linfty{Du}{V}, \linfty{Dv}{V}\}.
\]
\end{lem}

We proceed with:
\begin{proof}[Proof of Proposition \ref{compare to constant coef}]
We will  establish the estimate  
\[
\sup_{B_r}\left( u -\tu \right) \leq Cr^{2+\alpha};
\]
the proof of the estimate on $\sup_{B_r}\left( \tu -u\right) $ is similar. Throughout this proof we will use $C_i$ with $i=1,2,...$ to denote generic constants that may depend on $\lambda$, $\Lambda$, $n$, $\kappa$, $\linfty{f}{U}$, $\conealpha{g}{\bdry U}$, $\diam U$ and the regularity of $\bdry U$. 
We take $\alpha$ to be the exponent given by Proposition \ref{c one alpha prop} and $C(n)$ to be the constant from Lemma \ref{lem compare}. We define the constants $r_0$ and $\theta_0$ by
\[
r_0=\min \left\{\frac{\lambda}{2\kappa C(n)}, 1 \right\}
\]
and  
\[
\theta_0 = \frac{2C(n)^2\kappa^2}{\lambda}.
\]

We take $0< r\leq r_0$. 
We define the parameter $\theta$ by,
\[
\theta= \theta_0+r^{1-\alpha}\linfty{Df}{U}+ r^{1-\alpha}\kappa,
\]
and the function $\bu$ to be the solution of
\begin{equation}
\left\{\begin{array}{l l}
F(D^2 \bu, x_0)=f(x_0)-\theta r^{\alpha} &\quad \text{ in }B_r(x_0),\\
\bu=u &\quad  \text{ on }\bdry B_r(x_0).
\end{array} \right.
\end{equation}

Next, for $y\in B_1$, we define the rescaled functions $u^r$ and $\bu^r$ by,
\begin{equation}
\label{ur def}
u^r(y)=\frac{u(yr+x_0)-u(x_0) -Du(x_0)\cdot (yr) }{r^{1+\alpha}}
\end{equation}
and
\begin{equation}
\label{bur def}
\bu^r(y)=\frac{\bu(yr+x_0)-u(x_0) -Du(x_0)\cdot (yr) }{r^{1+\alpha}}.
\end{equation}
We denote by $\tilde{F}$ the rescaled nonlinearity,
\[
\tilde{F}(X,y) = r^{1-\alpha}F(r^{\alpha-1}X, ry+x_0),
\]
and define $\tilde{f}(y)=r^{1-\alpha}f(ry+x_0)$. 
The nonlinearity $\tilde{F}$ is uniformly elliptic with the same ellipticity constants as $F$. Moreover, since $F$ satisfies (F\ref{F lip}), we have,  for any $y_1$ and $y_2$ in $B_1$,
\begin{equation}
\label{prop of tildF}
|\tilde{F}(X,y_1)-\tilde{F}(X, y_2)|\leq \kappa (r||X||+r^{2-\alpha}).
\end{equation}
In addition, we have  the following estimate on  $D\tilde{f}$, which follows simply from the definition of $\tilde{f}$:
 \begin{equation}
 \label{est df}
 \linfty{D\tilde{f}}{B_1} = r^{2-\alpha}\linfty{Df}{B_r(x_0)} \leq r^{2-\alpha}\linfty{Df}{U}.
 \end{equation}

The definitions of $\tilde{F}$, $\tilde{f}$, $u^r$ and $\bar{u}^r$ imply that $u^r$ is a solution of
\[
\tilde{F}(D^2u^r, y)=\tilde{f}(y) \text{ in }B_1,
\]
and that $\bu^r(x)$ is a solution of 
\begin{equation}
\label{eqn:bu^r}
\left\{\begin{array}{l l}
\tilde{F}(D^2 \bu^r, 0)=\tilde{f}(0)-\theta r &\quad \text{ in }B_1,\\
\bu^r=u^r &\quad \text{ on }\bdry B_1.
\end{array}\right.
\end{equation}

We claim
\begin{equation}
\label{bound on B1}
\sup_{B_1} (u^r-\bu^r )\leq 4(\linfty{D\bu^r}{B_1}^2+\linfty{Du^r}{B_1}^2)r,
\end{equation}
which we will prove by contradiction. To this end, we assume,
\begin{equation}
\label{not bound on B1}
\sup_{B_1} (u^r-\bu^r )> 4(\linfty{D\bu^r}{B_1}^2+\linfty{Du^r}{B_1}^2)r.
\end{equation}
We double variables and consider
\begin{equation}
\label{double var r}
\sup_{x, y \in B_1} \left( u^r(x)-\bu^r(y) - \frac{r^{-1}}{2}|x-y|^2\right).
\end{equation}
The quantity in the previous line is larger than the difference between $u^r(x)$ and $\bu^r(x)$ for any $x$ in $B_1$; hence,
\[
\sup_{x, y \in B_1} \left( u^r(x)-\bu^r(y) - \frac{r^{-1}}{2}|x-y|^2\right) \geq \sup_{B_1}( u^r-\bu^r).
\]
Next, we use (\ref{not bound on B1}) to bound the right-hand side of the previous line from below and obtain,
\[
\sup_{x, y \in B_1} \left( u^r(x)-\bu^r(y) - \frac{r^{-1}}{2}|x-y|^2\right)> 4(\linfty{D\bu^r}{B_1}^2+\linfty{Du^r}{B_1}^2)r.
\]
Let us denote by $(x_r, y_r)$ a point in $\bar{B}_1\times \bar{B_1}$ where the supremum is achieved in (\ref{double var r}). The  previous line, together with  Lemma \ref{doubling variables property} applied with $V=B_1$ and $a=r^{-1}$,  implies   $(x_r, y_r)\in B_1\times B_1$. In other words, $(x_r, y_r)$ is an \emph{interior} maximum of the quantity (\ref{double var r}). 

We will apply now Lemma \ref{lem compare} with $a=r^{-1}$ and with $B_1$, $\tilde{F}(M,x)$, $\tilde{F}(M,0)$, $\tilde{f}$, and $\tilde{f}(0)-\theta r$ instead of $V$, $F$, $G$, $f$, and $g$, respectively. We have just shown that $(x_r,y_r)\in B_1\times B_1$. 
We now take
\[
t=\kappa r \text{ and } s=\kappa r^{2-\alpha}
\]
and verify  the two remaining hypotheses of Lemma \ref{lem compare}. Our choice of $r\leq r_0$  implies that  $t$ satisfies (\ref{assume t}). Let us now first verify that (\ref{assume eqn}) holds: assume $M, N\in \mathcal{S}_n$ with $M\leq N$. Then, by (\ref{prop of tildF}) and the uniform ellipticity of $\tilde{F}$, we find
\begin{align*}
\tilde{F}(M,x_r)-\tilde{F}(N,0)& \leq \tilde{F}(M,0) +\kappa r ||X|| +\kappa r^{2-\alpha} -\tilde{F}(N,0)
\\&\leq \kappa r ||X|| +\kappa r^{2-\alpha} -\lambda||N-M||.
\end{align*}
Thus (\ref{assume eqn}) is satisfied. Hence, with our choices of $t$ and $s$, the conclusion (\ref{conclusion lem compare}) of Lemma \ref{lem compare} yields, 
\begin{equation}
\label{frf0}
\tilde{f}(x_r) -(\tilde{f}(0) -\theta r) - \kappa r^{2-\alpha} \leq  \frac{t^2C^2(n)r^{-1}}{2\lambda}= \frac{\kappa^2rC^2(n)}{2\lambda}.
\end{equation}
Since   $x_r\in B_1$ and $f$ is Lipschitz, we may estimate the difference of the first two terms of the left-hand side of (\ref{frf0}) by,
\[
\tilde{f}(x_r) -\tilde{f}(0) \geq - \linfty{D\tilde{f}}{B_1} \geq -r^{2-\alpha}\linfty{Df}{U},
\]
where the second inequality follows from (\ref{est df}). We use the previous line to  estimate the left-hand side of  (\ref{frf0}) from below and obtain,
\[
-r^{2-\alpha}\linfty{Df}{U} +\theta r- \kappa r^{2-\alpha} \leq \frac{\kappa^2rC^2(n)}{2\lambda}.
\]
We use the definition of $\theta$ to rewrite the second term on the left-hand side, and find,
\[
-r^{2-\alpha}\linfty{Df}{U} +r( \theta_0+r^{1-\alpha}\linfty{Df}{U}+ r^{1-\alpha}\kappa)- \kappa r^{2-\alpha}  \leq \frac{\kappa^2rC^2(n)}{2\lambda}.
\]
We notice that the left-hand side is simply $\theta_0r$. Thus, the previous line reads,
\[
\theta_0r \leq \frac{C(n)^2\kappa^2}{\lambda}r ,
\]
which is impossible, since $r$ is positive and we chose $\theta_0=2\frac{C(n)^2\kappa^2}{\lambda}$. We have obtained the desired contradiction;  therefore, (\ref{bound on B1}) must hold. In order to complete the proof of the proposition, it is now left to bound the terms $\linfty{D\bu^r}{B_1}$ and $\linfty{Du^r}{B_1}$ that appear on the right-hand side of (\ref{bound on B1}) and to ``undo" the rescaling. To this end, from the definition of $u^r$, we see
\begin{equation}
\label{coneur}
\conealpha{u^r}{B_1} = \conealpha{u}{B_r(x_0)}\leq \conealpha{u}{U}\leq C_2,
\end{equation}
where the second inequality follows from Proposition \ref{c one alpha prop}. Since $\bar{u}^r$ is the solution of (\ref{eqn:bu^r}), Proposition \ref{c one alpha prop} applied to $\bar{u}^r$ in $B_1$ implies that there exists $C>0$ that depends on $n, \lambda, \Lambda$ and $\kappa$ such that
\[
\conealpha{\bar{u}^r}{B_1}\leq C(\linfty{\tilde{f}}{B_1}+\theta r+\conealpha{u^r}{B_1}) =C(r^{1-\alpha}\linfty{f}{B_r(x_0)} + \theta r+C_2),
\]
where the equality follows since $\linfty{\tilde{f}}{B_1}=r^{1-\alpha}\linfty{f}{B_r(x_0)}$ and from the estimate (\ref{coneur}) on $\conealpha{u^r}{B_1}$. 
The two previous estimates imply, 
\[
\linfty{D\bu^r}{B_1}^2+\linfty{Du^r}{B_1}^2 \leq C_3.
\]
We use this to bound the right-hand side of (\ref{bound on B1}) from above and obtain
\begin{equation}
\label{bd B_1 two}
\sup_{B_1} (u^r-\bu^r ) \leq C_3r.
\end{equation}
We have defined $\bu^r$ and $u^r$   by, respectively,  (\ref{bur def}) and (\ref{ur def}). Thus, subtracting (\ref{bur def}) from (\ref{ur def}) yields,  for all $y\in B_1$, 
\[
u^r(y)-\bu^r(y) = \frac{1}{r^{1+\alpha}}\left(u(yr+x_0)-\bu(yr+x_0)\right). 
\]
Therefore, the left-hand side of (\ref{bd B_1 two}) is exactly $r^{-(1+\alpha)} \sup_{B_r(x_0)} (u-\bu)$, so we have,
\begin{equation*}
\frac{1}{r^{1+\alpha}}\sup_{B_r(x_0)} (u-\bu)\leq C_3r.
\end{equation*}
Upon multiplying both sides by $r^{1+\alpha}$ we find,
\begin{equation}
\label{almost compare}
\sup_{B_r(x_0)} (u-\bu) \leq C_3r^{2+\alpha}.
\end{equation}
It is left to ``replace" $\bu$ by $\tu$. For this we employ Lemma \ref{barrier} in $B_r$ with $c=\theta r^\alpha$. We obtain, 
\[
\sup_{B_r(x_0)} (\bu -\tu) \leq \frac{\theta r^{2+\alpha}}{2\lambda}.
\]
Together with the estimate (\ref{almost compare}), we find
\begin{align*}
\sup_{B_r(x_0)} (u-\tu)&\leq \sup_{B_r(x_0)} (u-\bu)+\sup_{B_r(x_0)} (\bu-\tu) \\&
\leq C_3r^{2+\alpha} + \frac{\theta r^{2+\alpha}}{2\lambda}= C_4r^{2+\alpha},
\end{align*}
 thus completing the proof of the proposition.
\end{proof}

The following Corollary follows directly from Proposition \ref{compare to constant coef}, Proposition \ref{interior} and Proposition \ref{c one alpha prop}. 
\begin{cor}
\label{cor compare}
Assume (U\ref{assum U}), (F\ref{ellipticity}), (F\ref{F lip}), (G\ref{f assumption}), (G\ref{g c one alpha}) and $F(0,x)\equiv 0$.  Assume  $u$ is the viscosity solution of (\ref{eqn for u}). There exists a universal  constant $\alpha$ and positive constants $r_0=r_0(n, \lambda, \Lambda,\kappa)$ and $C$ that depends on $n$, $\lambda$, $\Lambda$, $\kappa$, $\diam U$, $\czeroone{f}{U}$, the regularity of $\bdry{U}$ and $\conegamma{g}{\bdry U}$, such that, if $\tu$ is the viscosity solution of
\begin{equation*}
\left\{\begin{array}{l l}
F(D^2 \tu, x_0)=f(x_0) &\quad \text{ in }B_{2r}(x_0),\\
\tu(x)=u(x) &\quad \text{ for }x\in\bdry B_{2r}(x_0),
\end{array}\right.
\end{equation*}
then
\begin{equation*}
\conealpha{\tu}{B_r(x_0)}^*\leq C
\end{equation*}
and 
\begin{equation*}
\linfty{\tu - u}{B_{2r}(x_0)}\leq Cr^{2+\alpha}.
\end{equation*}
\end{cor}

\section{Perturbations of the nonlinearity}
\label{section ep}
Let us state precisely the definitions of the two perturbations of the nonlinearity $F$ that we mentioned in Section \ref{subsec idea}. 
\begin{defn}
\label{def ep}
For  $F(X, x) \in C( \mathcal{S}_n \times U)$ and $\ep>0$, we define $F_\ep$ and $F^\ep$ by,
\[
F_\ep(X,x)=\inf_{y\in B_\ep (x)\cap U}F(X,y),\    \text{ and } F^\ep (X,x) =\sup_{y\in B_\ep (x) \cap U}F(X,y) .
\]
For $f\in C(U)$ and $\ep>0$, we  define $f_\ep$ and $f^\ep$ by
\[
f_\ep(x)=\inf_{y\in B_\ep (x)\cap U} f(y),\    \text{ and } f^\ep(x)=\sup_{y\in B_\ep (x)\cap U} f(y).
\]
\end{defn}
We observe that if $F$ satisfies (F\ref{ellipticity}) and (F\ref{F lip}), then so do $F_\ep$ and $F^\ep$. Our use of these perturbations is inspired by \cite[Theorem 2.1]{Krylov 1999}, a proof of the existence of $C^{2,\alpha}$ approximate solutions for convex equations, which is a key step in Krylov's analysis of the convex/concave case.  

This section is devoted to the proof of:
\begin{prop}
\label{compare u and uep}
Assume that $u$ and $u_\ep$ are  the viscosity solution of, respectively, (\ref{eqn for u}) and 
\begin{equation}\left\{
\begin{array}{l l}
\label{eqn for uep}
F_\ep(D^2 u_\ep, x)=f^\ep(x) &\quad \text{ in }U,\\
u_\ep=g &\quad \text{ on }\bdry U.
\end{array}\right.
\end{equation}
There exist positive constants $\ep_0$ and $C$ that depend on $n$, $\lambda$, $\Lambda$, $\kappa$, $\conegamma{g}{\bdry U}$,  $\czeroone{f}{U}$, $\diam U$ and the regularity of $\bdry U$ such that
for $\ep \leq \ep_0$,  and for all $x\in U$, 
\[
0\leq u(x)-u_\ep (x)\leq C \ep.
\]
\end{prop} 
A similar statement holds for $F^\ep$ and $f_\ep$ instead of $F_\ep$ and $f^\ep$, except there the conclusion is,
\[
0\leq u_\ep(x)-u (x)\leq C \ep.
\]
The proof of Proposition \ref{compare u and uep} is  similar to the proof Proposition \ref{compare to constant coef} --  the key is the application of Lemma \ref{lem compare}.

\begin{proof}[Proof of Proposition \ref{compare u and uep}]
The definitions of $F_\ep$ and $f^\ep$ imply that  $u_\ep$ is a subsolution of (\ref{eqn for u}). Therefore, we have $u_\ep \leq u$ for all $\ep$. The remainder of the proof is devoted to establishing an upper bound on $u-u_\ep$. 

By Proposition \ref{c one alpha prop}, there exists a constant $\bar{C}$ that depends on  $\lambda$, $\Lambda$, $\kappa$, $\diam U$ and the regularity of $\bdry U$ such that 
\begin{equation}
\label{estDuU}
\linfty{Du}{U}\leq \bar{C}(\conegamma{g}{\bdry U}+\linfty{f}{U}).
\end{equation} 
Let $C(n)$ be the constant from Lemma \ref{lem compare}. We define the constant $\ep_0$ by
\[
\ep_0= \frac{\lambda}{2\kappa C(n)(2\bar{C} (\conegamma{g}{\bdry U}+\linfty{f}{U})+1)},
\]
We remark that $\ep_0$ does not depend on $u$, and, according to (\ref{estDuU}), satisfies,
\begin{equation}
\label{ep0}
\ep_0 \leq  \frac{\lambda}{2\kappa C(n)(2\linfty{Du}{U} +1)}.
\end{equation}
We fix $\ep \leq \ep_0$ and introduce the parameters $\theta_0$ and $\theta$:
\[
\theta_0 = \ep\frac{ \kappa^2  ( 2\linfty{Du}{U} +1) ^2}{\lambda}
\]
and
\[
\theta= \theta_0+\ep (\kappa +\linfty{Df}{U}) ( 2\linfty{Du}{U} +1).
\]
We define $\tu_\ep$ to be the viscosity solution of 
\begin{equation*}\left\{
\begin{array}{l l}
F_\ep(D^2 \tu_\ep, x)=f^\ep(x)-\theta &\quad \text{ in }U,\\
\tu_\ep=g &\quad \text{ on }\bdry U.
\end{array}\right.
\end{equation*}
We claim,
\begin{equation}
\label{bd u-uep}
\sup_U (u - \tu_\ep)\leq 4(\linfty{D\tu_\ep}{U}^2 + \linfty{Du}{U}^2)\ep,
\end{equation}
which we will prove by contradiction. To this end, we assume that (\ref{bd u-uep}) does not hold, so that instead we have, 
\begin{equation}
\label{not bd u-uep}
\sup_U (u - \tu_\ep)> 4(\linfty{D\tu_\ep}{U}^2 + \linfty{Du}{U}^2)\ep.
\end{equation}
We double variables and consider
\begin{equation}
\label{doublevarep}
\sup_{x,y\in U} \left( u(x)-\tu_\ep(y) - \ep^{-1}\frac{|x-y|^2}{2}\right).
\end{equation}
We have that the quantity in the previous line is greater than the supremum in $x$ of $u(x)-\tu_\ep(x)$:
\[
\sup_{x,y\in U} \left( u(x)-\tu_\ep(y) - \frac{1}{\ep}\frac{|x-y|^2}{2}\right)\geq \sup_{U} (u-\tu_\ep).
\]
We use (\ref{not bd u-uep}) to bound the right-hand side of the previous line from below and obtain,
\[
\sup_{x,y\in U} \left( u(x)-\tu_\ep(y) - \frac{1}{\ep}\frac{|x-y|^2}{2}\right)\geq 4(\linfty{D\tu_\ep}{U}^2 + \linfty{Du}{U}^2)\ep.
\]
Let us denote by $(x_\ep, y_\ep)$ a point where the supremum is achieved in the quantity (\ref{doublevarep}). 
Together with the previous line,  Lemma  \ref{doubling variables property}  applied with $a=\ep^{-1}$ implies that $x_\ep$ and $y_\ep$ are both contained in the interior of $U$.

We will apply Lemma \ref{lem compare} with $a=\ep^{-1}$, and $F_\ep$ and $(f^\ep(x)-\theta)$ instead of $G$ and $g(x)$. 
It is left to chose the parameters $t$ and $s$ appropriately and verify the remaining  hypotheses of Lemma \ref{lem compare}.  To this end, let $M, N\in \mathcal{S}_n$ be such that $M\leq N$.  We denote by $y_\ep^*$  a point where the supremum is achieved in the definition of $F_\ep(N, y_\ep^*)$. In particular, we have
\begin{equation}
\label{yepyepstar}
|y_\ep^* -y_\ep|\leq \ep.
\end{equation}
 We use the definition of $y_\ep^*$ and the property (F\ref{F lip}) of $F$ to find,
\begin{align*}
F(M, x_\ep)-F_\ep(N,y_\ep)&= F(M, x_\ep)-F(N, y_\ep^*)\\
&\leq F(M, y_\ep^*) +\kappa |x_\ep-y_\ep^*| (||M||+1)- F(N, y_\ep^*).
\end{align*}
We use our assumption $M\leq N$ and the uniform ellipticity of $F$ to bound the right-hand side of the previous line from above and obtain,
\begin{equation}
\label{FMFN}
F(M, x_\ep)-F_\ep(N,y_\ep)
\leq \kappa |x_\ep-y_\ep^*| (||M||+1) -\lambda||N-M||.
\end{equation}
We will now estimate $|x_\ep-y_\ep^*|$. 
According to the triangle inequality, Lemma \ref{doubling variables property} and the estimate (\ref{yepyepstar}), we have
\begin{equation}
\label{xepyep}
|x_\ep - y_\ep^*|\leq |x_\ep-y_\ep| + |y_\ep-y_\ep^*|\leq 2\ep\linfty{Du}{U} +\ep= \ep(2\linfty{Du}{U}+1).
\end{equation}
We set $s$ and $t$ to be,
\[
s=t= \ep \kappa (2\linfty{Du}{U}+1).
\]
Together with the estimate (\ref{xepyep}), this implies,
\[
|x_\ep - y_\ep^*|\leq t\kappa^{-1}=s\kappa^{-1}.
\]
We use the previous line to bound from above the first term on the right-hand side of (\ref{FMFN}) and obtain,
\[
F(M, x_\ep)-F_\ep(N, y_\ep) \leq t||M||+s-\lambda||N-M||.
\]
And, by our choice of $\ep\leq \ep_0$, we have that $t$ satisfies (\ref{assume t}). Thus, the hypotheses (\ref{assume eqn}) and (\ref{assume t}) of Lemma \ref{lem compare} are satisfied, and  we obtain that (\ref{conclusion lem compare}) holds. In our situation, (\ref{conclusion lem compare}) reads,
\begin{equation}
\label{ep consequence}
f(x_\ep)-(f^\ep(y_\ep)-\theta)-\ep \kappa (2\linfty{Du}{U}+1)\leq \frac{t^2 C^2(n) \ep^{-1}}{2\lambda} =\frac{\ep \kappa^2(2\linfty{Du}{U}+1)^2}{2\lambda},
\end{equation}
where the equality follows by our choice of $t$.  
Next we will use that $f$ is Lipschitz to  bound the left-hand side of (\ref{ep consequence}) from below. We denote by $y_\ep^+$  a point where the supremum is achieved in the definition of $f^\ep(y_\ep)$. We have,
\begin{equation}
\label{yepplus}
|y_\ep^+-y_\ep|\leq \ep.
\end{equation}
According to the triangle inequality, the estimate on $ |x_\ep-y_\ep| $ of Lemma \ref{doubling variables property}, and (\ref{yepplus}), we have 
\begin{equation}
\label{est xep yep}
|x_\ep - y_\ep^+|\leq |x_\ep-y_\ep| + |y_\ep-y_\ep^+|\leq 2\ep\linfty{Du}{U} +\ep= \ep(2\linfty{Du}{U}+1).
\end{equation}
Since $f$ is Lipschitz, we find
\begin{equation}
\label{fxepfyep}
f(x_\ep)-f^\ep(y_\ep)=f(x_\ep)-f(y_\ep^+) \geq -\linfty{Df}{U}|x_\ep-y_\ep^+|\geq -\ep \linfty{Df}{U} (2\linfty{Du}{U}+1),
\end{equation}
where the second inequality follows from (\ref{est xep yep}). 
We use the previous line to estimate the left-hand side of (\ref{ep consequence}) from below and find,
\[
 - \ep \linfty{Df}{U} (2\linfty{Du}{U}+1))+\theta-\ep \kappa (2\linfty{Du}{U}+1)\leq \frac{\ep \kappa^2(2\linfty{Du}{U}+1)^2}{2\lambda},
\]
which, upon rearranging the left-hand side becomes,
\[
\theta - (\ep \linfty{Df}{U}+\kappa) (2\linfty{Du}{U}+1)\leq  \frac{\ep \kappa^2(2\linfty{Du}{U}+1)^2}{2\lambda}.
\]
According to the definition of the parameter $\theta$, the left-hand side of the previous line is simply $\theta_0$. Hence we obtain,
\[
\theta_0\leq \frac{\ep \kappa^2(2\linfty{Du}{U}+1)^2}{2\lambda},
\]
which contradicts our choice of $\theta_0$. Therefore (\ref{bd u-uep}) holds. We use Proposition (\ref{c one alpha prop}) to bound the right-hand side of (\ref{bd u-uep}) from above, and find, for some constant $C$ that depends on $n, \lambda, \kappa$,  $\conegamma{g}{U}$, $\diam U$ and the regularity of $\bdry U$,
\begin{equation}
\label{u tu}
\sup_U (u-\tu_\ep) \leq C\ep.
\end{equation}
Let us now ``replace" $\tu_\ep$ by $u_\ep$. We will emply Lemma \ref{barrier} with $c=\theta$. We find,  
\begin{equation}
\label{effect of theta}
\sup_U (\tu_\ep-u_\ep) \leq \frac{(\diam U)^2\theta}{2\lambda}.
\end{equation}
From (\ref{u tu}), (\ref{effect of theta}), and the definitions of $\theta$ and $\theta_0$, we conclude 
\[
\sup_U (u-u_\ep) \leq \sup_U(u-\tu_\ep)+\sup_U(\tu_\ep-u_\ep) \leq \sup_U (u - \tu_\ep) +\frac{(\diam U)^2\theta}{2\lambda} \leq \ep C,
\]
where $C$ depends on $n, \lambda, \kappa$, $\linfty{Df}{U}$, $\conegamma{g}{U}$, $\diam U$ and the regularity of $\bdry U$.
\end{proof}

\section{An elementary lemma}
\label{elem lemm sec}
In this section we establish an elementary lemma that plays an important role in the proof of our main result.
 \begin{lem}
\label{elem lem}
Assume $w\in C^{0,\eta}(U)$ is such that $w\leq 0$ on $\bdry U$ and $\sup_U w>0$.  Then for any positive $m$  with $m\leq \sup_U w$, there exists $x_0\in U$ with 
\begin{equation}
\label{where x0 is}
d(x_0, \bdry U) \geq \left(\frac{m}{2\etasemi{w}{U}}\right)^{1/\eta}
\end{equation}
 and an affine function $l(x)$ with $l(x_0)=w(x_0)$ and
\begin{equation}
\label{elemlem conc}
w(x)\leq l(x)-\frac{m}{2(\diam U)^2}|x-x_0|^2
\end{equation}
for all $x\in U$.
\end{lem}

\begin{proof}
Let $w\in C^{0,\eta}(U)$ be as in the statement of the lemma, and let us take $m$ with
\begin{equation*}
m\leq \sup_U w.
\end{equation*}
Let us use $R$ to  denote $R=\diam U$. We fix some $y\in U$. The function
\begin{equation}
\label{map elem}
x\mapsto -\frac{m}{2R^2}|x-y|^2 -w(x)
\end{equation}
achieves its minimum on $\bar{U}$ at some $x_0\in \bar{U}$. (We point out that $x_0$ is a point where $w$ is touched from \emph{above} by a \emph{concave} paraboloid of opening $\frac{m}{2R^2}$.) Thus, for all $x\in U$, we have
\begin{align*}
w(x_0)&\geq w(x) + \frac{m}{2R^2}|x-y|^2-\frac{m}{2R^2}|x_0-y|^2\\
&\geq w(x) -\frac{m}{2R^2}|x_0-y|^2.
\end{align*}
We take the supremum over $x\in U$ of both sides and obtain,
\[
w(x_0)\geq \sup_{x\in U} w(x) - \frac{m}{2R^2}|x_0-y|^2 .
\]
The definition of $R$ implies $|x_0-y|^2\leq R^2$. We use this, and that we have $m\leq \sup_U w$, to estimate the right-hand side of the previous line from below. We obtain,
\begin{equation}
\label{wxo}
w(x_0) \geq m - \frac{m}{2}=\frac{m}{2}.
\end{equation}
Since $w\in C^{0,\eta}(U)$ and $w\leq 0$ on $\bdry U$,  we obtain an upper bound on $w(x_0)$ in terms of the distance between $x_0$ and the boundary of $U$:
\[
w(x_0) \leq d(x_0, \bdry U)^\eta \etasemi{w}{U}.
\]
We use this to bound the left-hand side of (\ref{wxo}) from above and obtain,
\[
 d(x_0, \bdry U)^\eta  \etasemi{w}{U}\geq \frac{m}{2},
 \]
and so the estimate (\ref{where x0 is}) follows.

Let us now establish (\ref{elemlem conc}). Since $x_0$ is a minimum of the map (\ref{map elem}), we have that for all $x\in U$,
\begin{align*}
w(x)&\leq w(x_0) + \frac{m}{2R^2}\left(|x_0-y|^2-|x-y|^2\right).
\end{align*}
We have $|x-y|^2=|x-x_0|^2+ 2\langle x-x_0,x_0-y\rangle + |x_0-y|^2$, so we find,
\begin{align*}
w(x)&\leq w(x_0) +\frac{m}{2R^2}\left(-|x_0-x|^2-2\langle x-x_0,x_0-y\rangle\right),
\end{align*}
which yields (\ref{elemlem conc}) upon setting $l(x)=w(x_0) -\frac{m}{R^2}\langle x-x_0,x_0-y\rangle$.
\end{proof}

\section{A key estimate between solutions and sufficiently regular $\delta$-solutions}
\label{main lem sec}
In this section we state and prove Proposition \ref{main lem}, a key part of the proofs of the two main results, Theorems \ref{vague main thm} and \ref{vague main thm h}. 
Roughly, Proposition \ref{main lem} says that if $w$, $u$ and $v$ are, respectively, a $\delta$-subsolution, a solution and a $\delta$-supersolution  on a ball of radius $2r$ and satisfy,
\[
w\leq u\leq v
\]
on the boundary of the  ball of radius $r/2$, then we have
\[
w -\tilde{c}\delta^{\tilde{\alpha}}r^2\leq u \leq v+\tilde{c}\delta^{\tilde{\alpha}}r^2,
\]
on the interior of  the ball of radius $r/2$. The key ingredient in the proof is Proposition \ref{thm A};  the main difficulty is that we assume $w$, $u$ and $v$ satisfy \emph{different} equations. In particular, $u$ satisfies an equation with frozen coefficients. 
We introduce the  constants
\begin{align*}
&\zeta = \frac{2\sigma}{12\sigma+9n},  \   \   \theta= \frac{\zeta}{4(2-\alpha)}, \   \    \alpha_1 = \frac{\zeta}{4},
\end{align*}
where $\sigma$ is the exponent from Proposition \ref{thm A} and $\alpha$ is the exponent from  Proposition \ref{c one alpha prop}.  We remark that since $\sigma$ and $\alpha$ are universal, then so are $\zeta$, $\theta$, and $\alpha_1$.

\begin{prop}
\label{main lem}
Assume (F\ref{ellipticity}), (F\ref{F lip}), and $F(0,x)\equiv 0$. Suppose $x_0\in U$ and $C_1$ and $\nu$ are positive constants. For $\delta>0$, we define the quantities $r$ and $\ep$ by
\[
r=\delta^\theta,\text{  and  }\ep = \nu+\delta^\theta=\nu +r.
\]
There exist a universal constant   $\tilde{\alpha}$ and positive constants $\tilde{c}$, $\tilde{\delta}$ that depend on $C_1$, $n$, $\lambda$, and $\Lambda$ such that if $\delta\leq \tilde{\delta}$ and, 
\begin{enumerate}
\item   if $v$ is a $\delta$-supersolution  of
\begin{equation}
\label{eqn for v}
F_{\nu}(D^2v, x)=f^\nu(x) \text{ in } B_{2r}(x_0)
\end{equation}
and $\tu_\ep$ is a viscosity solution of 
\begin{equation}
\label{eqn for tuep}
F_\ep(D^2 \tu_\ep, x_0)=f^\ep(x) \text{ in }B_{2r}(x_0) 
\end{equation}
that satisfy 
\begin{enumerate}
\item
\label{A1}
 $\displaystyle \conealpha{\tu_\ep}{B_{r}(x_0)}^*\leq C_1$,
\item
\label{A1.5}
$\displaystyle D^2v(x) \leq \delta^{-\zeta}I$ in the sense of distributions for all $x\in B_{2r}(x_0)$, and
\item
\label{A2}
$\displaystyle \sup_{\bdry B_{r/2}(x_0)} (\tu_\ep - v)\leq 0$,
\end{enumerate}
then
\begin{equation}
\label{bound on u-v}
\sup_{B_{r/2}(x_0)} (\tu_\ep - v)\leq \tilde{c}\delta^{\tilde{\alpha}}r^2; 
\end{equation}
\item [(2)]  if $w$ is a $\delta$-subsolution of $F^{\nu}(D^2w, x)=f_\nu(x)$ in $B_{2r}(x_0)$ and $\tilde{u}^\ep$ a solution of $F^\ep(D^2 \tu^\ep, x_0)=f_\ep(x_0)$ in $B_{2r}(x_0)$ that satisfy $\conealpha{\tu^\ep}{B_{r}(x_0)}^*\leq C_1$, $D^2w(x)\geq -\delta^{-\zeta}I$ in the sense of distributions for all $x\in B_{2r}(x_0)$,
and $
\sup_{\bdry B_{r/2}(x_0)} ( w- \tu^\ep )\leq 0$, 
then, 
\[
\sup_{B_{r/2}(x_0)}( w- \tu^\ep) \leq \tilde{c}\delta^{\tilde{\alpha}}r^2.
\]
\end{enumerate}
\end{prop}

\subsection{Outline of the proof of Proposition \ref{main lem}}
We outline the proof of the first part of Proposition \ref{main lem}; the proof of the second part is very similar. The main idea is to control the supremum of $(\tu_\ep-v)$ by the size of the contact set of  $(\tu_\ep-v)$ with its concave envelope. 

We regularize $\tu_\ep$ by sup-convolution.  Then we subtract a small quadratic in order to obtain a strict subsolution of (\ref{eqn for tuep}). Let us denote the resulting function by $\bu$. The assumptions (\ref{A2}) and (\ref{A1.5}) imply that  $\bu$ and $v$ satisfy the hypotheses of Proposition \ref{pre ABP}, so we find,
\[
\sup (\bu-v)\leq Cr\left(\int_{\{\bu-v=\Gamma_{\bu-v}\}}|\det D^2\Gamma_{\bu-v}|\right)^{\frac{1}{n}}.
\] 
We proceed by contradiction and assume $\sup (\bu-v)$ is ``large". The previous estimate, together with an upper bound on $|\det D^2\Gamma_{\bu-v}|$,  thus implies that the contact set $\{\bu-v=\Gamma_{\bu-v}\}$ is ``large" as well. 

The key part of our argument is Proposition \ref{thm A}, which says that there is a set $A_t^+$ on which $\bu$ is very close to being a paraboloid. Moreover, Proposition \ref{thm A} provides a lower bound on the size of this good set $A_t^+$, which, together with the lower bound on the size of $\{\bu-v=\Gamma_{\bu-v}\}$, allows us to find a point $x$ in their intersection. We formulate this  as Lemma \ref{helper lemma} below. 

We show that if $\bu$ is touched from below by a paraboloid at a point in the contact set $\{\bu-v=\Gamma_{\bu-v}\}$, then $v$ is touched from below a paraboloid of the same opening. This allows us to use the fact that $v$ is a $\delta$-supersolution of (\ref{eqn for v}) and obtain the desired contradiction.

\subsection{An auxiliary  lemma}
Before stating Lemma \ref{helper lemma}, we introduce the function $\omega: \rr^+\rightarrow [0,1]$, defined by
\[
\omega(c) = \min \left\{ 1, \left(\frac{1}{4(1+2c)}\right)^{\frac{1}{\zeta-2\theta}}\right\}.
\]
We point out that, since $\zeta - 2\theta\geq \zeta/2>0$, we have
\begin{equation}
\label{omega property}
t\leq \omega(c) \Rightarrow 2t^{\zeta-\theta} (1+2c)\leq \frac{t^\theta}{2} .
\end{equation}
\begin{lem}
\label{helper lemma}
Under the assumptions of Proposition \ref{main lem}, there exists a universal constant $\bar{C}$ and a positive constant $c_1$ that depends on $n$, $\lambda$, $\Lambda$ and $C_1$ such that if $\mathcal{C}$ is a subset of $B_{r/2}(x_0)$ with 
\begin{equation*}
|\mathcal{C}|\geq c_1\delta^{n(\alpha_1+\zeta)}r^n,
\end{equation*}
and  $\delta\leq \omega(C_1)$, then there exists a point $x\in \mathcal{C}$ and a paraboloid $P$ such that
\begin{enumerate}
\item \label{Bdel} $B_\delta(x)\subset B_r(x_0)$,
\item \label{Px0=0} $P(x_0)=0$,
\item  \label{eq for P}
$F_\ep(D^2P, x) = f^\ep(x)$, and,
\item  \label{helper lem touch} for all $y\in B_\delta(x)$, we have,
\begin{equation}
\label{tu touched quad}
\tu^+_\ep (y) \geq \tu^+_\ep (x)+P(y) - \bar{C}\delta^{1/2}|x-y|^2.
\end{equation}
\end{enumerate}
\end{lem}

\subsection{Proof of Proposition \ref{main lem}}
We postpone the proof of Lemma \ref{helper lemma} and proceed with:
\begin{proof}[Proof of Proposition \ref{main lem}]
Throughout the proof of this proposition, we use $C$ to denote a generic constant that may change from line to line and  depends only on $\lambda$, $\Lambda$, $n$, $\kappa$, and $C_1$. 
We set
\[
\tilde{\delta}=\min \left\{\omega(C_1), \left(\frac{\Lambda}{2\lambda \bar{C}}\right)^4\right\}
\]
and take $\delta\leq \tilde{\delta}$. 
We will give the proof of part (1) of the proposition; the proof of part (2) is similar. We first regularize $\tu_\ep$ by taking sup-convolution:
\[
\tu_\ep^+(x)=\sup_y \left\{\tu_\ep(y)-\frac{|y-x|^2}{2\delta^{\zeta}}\right\}.
\]
Next we perturb $\tu_\ep^+$ by a small quadratic:
\[
\bu (x)= \tu^+_\ep(x) - \delta^{1/4}\left(\frac{r^2}{4} - |x-x_0|^2\right).
\]
Assumption (\ref{A1}) of this lemma implies $\linfty{\tu^\ep}{B_r(x_0)}\leq C_1$, so, by the properties of sup-convolution (item (\ref{item dist conv}) of Proposition \ref{prop conv}), we have
\[
\sup_{\bdry B_{r/2}(x_0)} (\tu_\ep^+ - v) \leq \sup_{\bdry B_{r/2}(x_0)} (\tu_\ep - v) +2\delta^{\zeta/2}\linfty{\tu_\ep}{B_r(x_0)}^{1/2} \leq 2\delta^{\zeta/2}  C_1^{1/2}.
\]
Since we have $\bu=\tu^+_\ep$ on $ \bdry B_{r/2}(x_0)$, we find
\[
\sup_{\bdry B_{r/2}(x_0)} (\bu - v) = \sup_{\bdry B_{r/2}(x_0)} (\tu_\ep^+ - v) \leq  2C_1^{1/2}\delta^{\zeta/2}.
\]
 Proposition \ref{prop conv} and the assumption (\ref{A1.5}) imply $D^2(\bu-v)(x)\geq -2\delta^{-\zeta}I$ in the sense of distributions for all $x\in B_r(x_0)$. Therefore, $ ( \bu - v - 2C_1^{1/2}\delta^{\zeta/2})$ satisfies the assumptions of  Proposition \ref{pre ABP}. We thus have
\begin{equation}
\label{est bar u v}
\sup_{B_{r/2}(x_0)} ( \bu - v - 2C_1^{1/2}\delta^{\zeta/2}) \leq Cr\left(\int_{\mathcal{C}}\det D^2\bar{\Gamma} \right)^{1/n},
\end{equation}
where $C$ is a universal constant, $\bar{\Gamma}$ is the concave envelope of $(\bu - v - 2C_1^{1/2}\delta^{\zeta/2})^+$ on $B_r(x_0)$, and $\mathcal{C}$ is the contact set of $(\bu - v - 2C_1^{1/2}\delta^{\zeta/2})^+$ with $\bar{\Gamma}(x)$:
\[
\mathcal{C} = \{x \in B_{r/2}(x_0): \   \bar{\Gamma}(x) =\bu(x) - v(x) - 2C_1^{1/2}\delta^{\zeta/2} \}
\]
 If $x\in \mathcal{C}$, then we have, for all $y\in B_r(x_0)$,
\[
\bar{\Gamma}(y)\geq \bu(y) - v(y) - 2C_1^{1/2}\delta^{\zeta/2}(y), 
\]
with equality holding at $x$. Since $D^2(\bu-v)(x)\geq -2\delta^{-\zeta}I$ in the sense of distributions for all $x\in B_r(x_0)$, we therefore find that 
\[
D^2\bar{\Gamma}(x)\geq -2\delta^{\zeta}I
\]
for all points $x\in \mathcal{C}$.  Moreover, $\bar{\Gamma}$ concave and, according to Proposition \ref{pre ABP}, is twice differentiable almost everywhere on $B_r(x_0)$. Therefore,  we obtain $|\det D^2\bar{\Gamma}(x)|\leq 2\delta^{-n\zeta}$ for almost every $x\in \mathcal{C}$. We use this to bound the right-hand side of (\ref{est bar u v}) from above and find,
\begin{equation}
\label{bd sup below}
\sup_{B_{r/2}(x_0)} (\bu - v - 2C_1^{1/2}\delta^{\zeta/2} )\leq Cr\left(\int_{\mathcal{C} }\delta^{-n\zeta}\right)^{1/n}.
\end{equation}
Let  $c_1$ be the constant given by Lemma \ref{helper lemma}. We proceed by  contradiction  and assume
\begin{equation}
\label{cont assump}
\sup_{B_{r/2}(x_0)} (\bu - v - 2C_1^{1/2}\delta^{\zeta/2} ) > c_1^{1/n} \delta^{\alpha_1}r^2.
\end{equation}
We use  (\ref{cont assump})  to estimate the left-hand side of  (\ref{bd sup below}) from below. We find,
\[
c_1^{1/n} \delta^{\alpha_1}r^2< Cr\left(\int_{\mathcal{C} }\delta^{-n\zeta}\right)^{1/n} \leq Cr|\mathcal{C}|^{1/n}\delta^{1/n}.
\]
Rearranging, we obtain,
\begin{equation}
\label{bound Gamma}
|\mathcal{C} |\geq c_1 \delta^{n(\alpha_1+\zeta)}r^n.
\end{equation}
By Lemma \ref{helper lemma}, there exists a point  $x\in \mathcal{C}$ with
\begin{equation}
 \label{Bdeluse} B_\delta(x)\subset B_r(x_0),
 \end{equation}
  and a paraboloid $P$ with $P(x_0)=0$ that satisfies
\begin{equation}\label{eq for Puse}
F_\ep(D^2P, x) = f^\ep(x),
\end{equation}
 and such that, for all $y\in B_\delta(x)$ we have,
\begin{equation}
\label{tu touched quad use}
\tu^+_\ep (y) \geq \tu^+_\ep (x)+P(y) - \bar{C}\delta^{1/2}|x-y|^2.
\end{equation}
Since $x$ is contained in the contact set $\mathcal{C}$, there exists an affine function $l(y)$ such that for all $y\in B_{r/2}(x_0)$,
\[
\bu(y) - v(y) \leq l(y) \text{ with equality holding at } x.
\]
Rearranging the previous inequality and using the definition of $\bu$ in terms of $\tu^+_\ep$ yields, for all $y\in B_{r/2}(x_0)$, 
\[
v(y)\geq \bu(y)-l(y) = \tu^+_\ep(y) - \delta^{1/4}\left(\frac{r^2}{4}-|y-x_0|^2\right)-l(y) \text{ with equality holding at }x.
\]
Since, according to (\ref{Bdeluse}) we have  $B_\delta(x)\subset B_r(x_0)$, we may use (\ref{tu touched quad use}) to bound the first term on the right-hand side of the previous line from below and find, for all $y\in  B_\delta(x)$,
\[
v(y)\geq  \tu^+_\ep (x)+P(y) - \bar{C}\delta^{1/2}|x-y|^2 - \delta^{1/4}\left(\frac{r^2}{4}-|y-x_0|^2\right)-l(y),
\]
with equality holding at $x$.  Since $v$ is a $\delta$-supersolution of (\ref{eqn for v}) on $B_{2r}(x_0)$, and the previous inequality holds for all $y\in  B_\delta(x)\subset B_r(x_0)$, we obtain,
\[
F_\nu(D^2P +\left(2\delta^{1/4}- 2\bar{C}\delta^{1/2}\right)I, x)\leq f^\nu(x).
\]
Let $x^*$ be a point at which the infimum is achieved in the definition of $F_\nu(D^2P +\left(\delta^{1/4} - Ct\delta\rho^{-1}\right)I, x)$ and let $x^+$ be a point at which the supremum is achieved in the definition of $f^\nu(x)$.
Then 
\begin{equation}
\label{xx*}
|x-x^*|\leq \nu, \    \    |x-x^+|\leq \nu,
\end{equation}
and
\[
F(D^2P +\left(2\delta^{1/4}- 2\bar{C}\delta^{1/2}\right)I, x^*) \leq f(x^+).
\]
Since $F$ is uniformly elliptic, the above implies 
\begin{equation}
\label{strict ineq}
F(D^2P , x^*) - 2\lambda \bar{C}\delta^{1/2} \leq f(x^+)-2\Lambda \delta^{1/4}.
\end{equation}
We subtract (\ref{strict ineq}) from (\ref{eq for Puse}) to obtain
\begin{equation}
\label{ineq 2}
F_\ep(D^2P, x_0) - F(D^2P , x^*) +2\lambda \bar{C}\delta^{1/2} \geq f^\ep(x)-f(x^+)+2\Lambda\delta^{1/4}. 
\end{equation}
By (\ref{xx*}) and our choice of $\ep$, we have
\[
|x^*-x_0| \leq |x^*-x|+|x-x_0| \leq\nu+r = \ep.
\]
Together with the definition of $F_\ep$ as the infimum of $F$ over balls of size $\ep$, the previous line implies, 
\[
F_\ep(D^2P, x_0) \leq F(D^2P , x^*).
\]
 Similarly, we have $|x^+-x_0|\leq \ep$, and so $f^\ep(x)-f(x^+)\geq 0$. Therefore, (\ref{ineq 2}) becomes
\[
2\Lambda \delta^{1/4}\leq 2\lambda \bar{C} \delta^{1/2}.
\]
Rearranging the previous line, we find,
\[
\delta^{1/4}\geq \frac{\Lambda}{\lambda \bar{C}}.
\]
But this contradicts our choice of  $\delta$. Therefore, (\ref{cont assump}) cannot hold, so we obtain,
\[
\sup_{B_{r/2}(x_0)} (\bu - v - 2C_1^{1/2}\delta^{\zeta/2}) \leq c_1^{1/n} \delta^{\alpha_1}r^2,
\]
which, upon rearranging becomes,
\[
\sup_{B_{r/2}(x_0)} (\bu - v ) \leq 2C_1^{1/2}\delta^{\zeta/2} +c_1^{1/n} \delta^{\alpha_1}r^2.
\]
Our choices of $r$ and $\alpha_1$ are such that $\alpha_1+2\theta \geq \zeta/2$, so that $\delta^{\zeta/2}\leq \delta^{\alpha_1+2\theta}= \delta^{\alpha_1}r^2$. Thus we obtain,
\begin{equation}
\label{almost there bd}
\sup_{B_{r/2}(x_0)} ( \bu - v )\leq (2C_1^{1/2}+ c_1^{1/n})\delta^{\alpha_1}r^2.
\end{equation}
In addition, according to the definition of $\bu$ in terms of $\tu_\ep^+$, we have,
\[
\tu_\ep^+(x)\leq \bu(x) + \delta^{\frac{1}{4}}r^2.
\]
Because $\tu_\ep^+$ is the sup-convolution of $\tu_\ep$,  we have $\tu_\ep\leq \tu_\ep^+$.  We use this, together with the previous line, and find, for all $x\in B_{r/2}(x_0)$,
\[
\tu_\ep(x)\leq\tu_\ep^+(x)\leq \bu(x) + \delta^{\frac{1}{4}}r^2.
\]
Together with the bound (\ref{almost there bd}) this implies:
\[
\sup_{B_{r/2}(x_0)} (\tu_\ep - v)\leq  \sup_{B_{r/2}(x_0)} (\bu + \delta^{\frac{1}{4}}r^2- v) \leq (2C_1+ c_1^{1/n})\delta^{\alpha_1}r^2 + \delta^{\frac{1}{4}}r^2.
\]
We take $\tilde{\alpha} = \min\left\{\alpha_1, 1/4\right\}$ and $\tilde{c} = 2C_1+ c_1^{1/n}+1$ to conclude.
\end{proof}

\subsection{Proof of  auxiliary lemma}
Lemma \ref{helper lemma} follows from Proposition \ref{thm A} by a covering argument, which we now describe. We cover $B_{r/2}(x_0)$ by balls $B^{\delta^{\zeta}}_{\rho/2}(x_i)$ (we are using the notation of Proposition \ref{thm A}), with the parameter $\rho$ properly chosen. By Proposition \ref{thm A}, $\tilde{u}_\ep$ has second order expansions with controlled error on large portions of each of the $B^{\delta^{\zeta}}_{\rho/2}(x_i)$. We  refer to such points as being in the ``good set" of $\tilde{u}_\ep$. We will use the lower bound on $|\mathcal{C}|$ to show that there is a point $x$ that is both in the good set of $\tilde{u}_\ep$ and in $\mathcal{C}$. 

\begin{proof}[Proof of Lemma \ref{helper lemma}]
Let us take $\mathcal{C}$ to be a subset of $B_{r/2}(x_0)$ that satisfies
\begin{equation}
\label{Clowbd}
|\mathcal{C}|\geq c_1\delta^{n(\alpha_1+\zeta)}r^n,
\end{equation}
where $c_1$ is specified in (\ref{chose c1}). Let us define the parameter $\rho$ as,
\[
\rho = 2\delta^{\zeta}(1+ 2\linfty{D\tu_\ep}{B_r(x_0)}).
\]
We recall some notation of Proposition \ref{thm A}: we will be using, for $x_i\in B_r(x_0)$, the set  $B^{\delta^{\zeta}}_{\rho}(x_i)$ given by
\[
B^{\delta^{\zeta}}_{\rho}(x_i) = B(x_i, \rho-2\delta^\zeta\linfty{D\tilde{u}_\ep}{B_\rho(x_i)}).
\]
We point out that $\rho$ depends on the  norm of $D\tu_\ep$ in $B_r(x_0)$, while the norm of  $D\tu_\ep$ in $B_\rho(x_i)$ appears in the definition of $B^{\delta^{\zeta}}_{\rho}(x_i)$.

\textbf{Step one.}  
Before proceeding with the proof, we establish several important relationships between the various sets we are using in this proof. We claim
\begin{equation}
\label{inside Brho}
\text{ if }x_i\in B_{r/2}(x_0), \text{ then }B(x_i, \delta^\zeta)\subset B_{\rho/2}^{\delta^\zeta}(x_i).
\end{equation}
We now proceed with verifying (\ref{inside Brho}).
By assumption (\ref{A1}), we have $\linfty{D\tu_\ep}{B_r(x_0)}\leq C_1r^{-1}$. We use this bound in the definition of $\rho$ to find 
\[
\rho  \leq 2\delta^\zeta(1+2C_1r^{-1})\leq 2\delta^\zeta r^{-1}(1+2C_1)=2\delta^{\zeta-2\theta}(1+2C_1).
\]
The second inequality follows since $0\leq r\leq 1$, and the equality holds by  our definition of $r$ as $r=\delta^\theta$. 
Since we assumed $\delta\leq \omega(C_1)$, the property (\ref{omega property}) of $\omega$ implies,
\[
2\delta^{\zeta-2\theta}(1+2C_1) \leq \frac{\delta^\theta}{2} = \frac{r}{2}.
\]
Thus we have 
\begin{equation}
\label{rho and r}
\rho\leq r/2.
\end{equation} 
Therefore, for any $x_i$, we have $B_{\rho}(x_i) \subset B_{r/2}(x_i)$. And, if $x_i\in B_{r/2}(x_0)$, we find that the inclusion $B_{r/2}(x_i)  \subset B_r(x_0)$ holds.  We summarize this as:
\begin{equation}
\label{upper on Brho}
B_{\rho}(x_i) \subset B_{r/2}(x_i)  \subset B_r(x_0).
\end{equation}
From the previous line we deduce,
\begin{equation}
\label{deriv balls}
\linfty{D\tu_\ep}{B_\rho(x_i)} \leq  \linfty{D\tu_\ep}{B_r(x_0)},
\end{equation}
which allows us to estimate the radius of $B_{\rho/2}^{\delta^\zeta}(x_i)$ from below:
\[
\frac{\rho}{2} - 2\delta^\zeta\linfty{D\tilde{u}_\ep}{B_\rho(x_i)} \geq \frac{\rho}{2} - 2\delta^\zeta\linfty{D\tilde{u}_\ep}{B_{r}(x_0)} =\delta^\zeta.
\]
Since the left-hand side of the previous line is exactly the radius of $B_{\rho/2}^{\delta^\zeta}(x_i)$, we find that (\ref{inside Brho}) holds.
Let us use (\ref{deriv balls}) one more time to obtain,
\[
\rho - 2\delta^\zeta\linfty{D\tilde{u}_\ep}{B_\rho(x_i)} \geq \rho - 2\delta^\zeta\linfty{D\tilde{u}_\ep}{B_{r}(x_0)} =2\delta^\zeta(1+\linfty{D\tilde{u}_\ep}{B_{r}(x)}). 
\]
We recognize the left-hand side of the previous line as exactly the radius of $B_{\rho}^{\delta^\zeta}(x_i)$. Therefore, we have
\begin{equation}
\label{Brho contains}
B(x_i, 2\delta^\zeta(1+\linfty{D\tilde{u}_\ep}{B_{r}(x_0)}))\subset B_{\rho}^{\delta^\zeta}(x_i).
\end{equation}
In addition, we claim 
\begin{equation}
\label{Bdelrhos}
\text{if $x_i\in B_{r/2}(x_0)$ and } x\in B_{\rho/2}^{\delta^\zeta}(x_i), \text{ then }B_\delta(x)\subset  B_{\rho}^{\delta^\zeta}(x_i)\subset B_{r}(x_0).
\end{equation}
To establish (\ref{Bdelrhos}), we take $y\in B_\delta(x)$. By the triangle inequality, we have,
\[
|x_i-y|\leq |x_i-x|+|x-y|.
\]
The radius of $B^{\delta^{\zeta}}_{\rho/2}(x_i)$ is less that $\rho/2= \delta^{\zeta}(1+ 2\linfty{D\tu_\ep}{B_r(x_0)})$. Since $x$ is contained in  $B_{\rho/2}^{\delta^\zeta}(x_i)$, we use this to bound from above the first term on the right-hand side of the previous line. Since $y\in B_\delta(x)$, the second term is bounded simply by $\delta$. Therefore we obtain,
\[
|x_i-y|\leq \delta^{\zeta}(1+ 2\linfty{D\tu_\ep}{B_r(x_0)}) +\delta.
\]
Since we have $\delta\leq \delta^\zeta$, we find,
\[
|x_i-y|\leq \delta^{\zeta}(2+ 2\linfty{D\tu_\ep}{B_r(x_0)}).
\]
Therefore, $y\in B(x_i, \delta^{\zeta}(2+ 2\linfty{D\tu_\ep}{B_r(x_0)})$. Together with (\ref{Brho contains}), this implies  $y\in B_{\rho}^{\delta^\zeta}(x_i)$. Since this holds for all $y\in B_\delta(x)$, we have established 
\[
B_\delta(x)\subset  B_{\rho}^{\delta^\zeta}(x_i).
\]
 Since we have $B_{\rho}^{\delta^\zeta}(x_i)\subset B_\rho(x_i)$ and (\ref{upper on Brho}) says $B_\rho(x_i)\subset B_r(x_0)$, we find
 \[
  B_{\rho}^{\delta^\zeta}(x_i)\subset B_r(x_0).
 \]
  Thus (\ref{Bdelrhos}) holds.

\textbf{Step two.} 
The collection $\left\{B_{\rho/2}^{\delta^\zeta}(x_i) : \  x_i \in B_{r/2}(x_0)\right\}$ covers $B_{r/2}(x_0)$; we seek to extract a finite subcover. Although the radius of each ball $B_{\rho/2}^{\delta^\zeta}(x_i)$ depends on $x_i$, the estimate (\ref{inside Brho}) provides a lower bound on these radii that is uniform in $x_i$. 
Therefore,  there exists a finite collection $\left\{B^{\delta^\zeta}_{\rho/2}(x_i): \  x_1, ..., x_L\in B_{r/2}(x_0)\right\}$ that covers $B_{r/2}(x_0)$, where $L=Cr^n\delta^{-n\zeta}$, and $C$ is a universal constant. 
There must be one ball $B^{\delta^\zeta}_{\rho/2}(x_I)$ with
\[
|\mathcal{C} \cap B^{\delta^\zeta}_{\rho/2}(x_I)|\geq \frac{|\mathcal{C}|}{L}.
\]
 We use the lower bound (\ref{Clowbd}) on the size of $\mathcal{C}$ to estimate the right-hand side of the previous line from below. We find,
\begin{equation}
\label{lbC}
|\mathcal{C} \cap B^{\delta^\zeta}_{\rho/2}(x_I)| 
 \geq \frac{c_1 \delta^{n(\alpha_1+\zeta)}r^n}{Cr^n\delta^{-n\zeta}}=C_2c_1\delta^{n\alpha_1+2n\zeta} .
\end{equation}
Here $C_2$ is universal.

According to (\ref{Bdelrhos}), we have,
\[
B^{\delta^\zeta}_{\rho}(x_I)\subset B_r(x_0),
\]
and so $\tu_\ep^+$ is satisfies the equation (\ref{eqn for tuep}) in $B^{\delta^\zeta}_{\rho}(x_I)$. Therefore, $\tu_\ep^+$ satisfies the hypotheses of Proposition \ref{thm A} in $B^{\delta^\zeta}_{\rho}(x_I)$ with  $F_\ep$ instead of $F$ and with constant right-hand side. Thus, for all $t\geq t_0$, where $t_0$ is a universal constant, there exist sets $A^+_t$ (the ``good sets" of  $\tilde{u}_\ep$) that satisfy
\begin{equation}
\label{size bad}
|B^{\delta^{\zeta}}_{\rho/2}(x_I) \setminus A^{+}_t| \leq C_3\rho^{n-\sigma}\linfty{D\tu_\ep}{B_\rho(x_I)}^\sigma t^{-\sigma} ,
\end{equation}
where $C_3$ and $\sigma$ are universal. Before proceeding, we will bound the right-hand side of the previous line from above. According to (\ref{deriv balls}), we have $\linfty{D\tu_\ep}{B_\rho(x_I)}\leq \linfty{D\tu_\ep}{B_r(x_0)}$. Hence,  assumption (\ref{A1}) implies  
\[
\linfty{D\tu_\ep}{B_\rho(x_I)}\leq C_1r^{-1}.
\]
The estimate (\ref{rho and r}) implies  $r^{-1}\leq \rho^{-1}/2$. We use this, together with the previous line, to bound the right-hand side of (\ref{size bad}) from above and obtain,
\begin{equation}
\label{size bad use}
|B^{\delta^{\zeta}}_{\rho/2}(x_I) \setminus A^{+}_t| \leq C_4\rho^{n-2\sigma} t^{-\sigma} ,
\end{equation}
where $C_4$ depends on $n$, $\lambda$, $\Lambda$  and $C_1$. 
We take $t$ to be, 
\begin{align*}
t=\left(\frac{2C_4\rho^{n-2\sigma}}{C_2c_1\delta^{n\alpha_1+2n\zeta }}\right)^{\frac{1}{\sigma}} +t_0.
\end{align*}
We use this choice of $t$ to bound the right-hand side of (\ref{size bad use}) from above and obtain,
\[
|B^{\delta^{\zeta}}_{\rho/2}(x_I) \setminus A^{+}_t| < \frac{1}{2}C_2c_1\delta^{n\alpha_1+2n\zeta } \leq \frac{1}{2} |\mathcal{C} \cap B^{\delta^\zeta}_{\rho/2}(x_I)|,
\]
where the second inequality holds by the  bound (\ref{lbC}) that we have just established. 
Therefore, there exists a point $x\in \mathcal{C} \cap  B^{\delta^\zeta}_{\rho/2}(x_I) \cap A_t^+$. Together with (\ref{Bdelrhos}), this implies $B_\delta(x)\subset B_r(x_0)$, so we have that the desired item (\ref{Bdel}) holds. The definition of $A_t^+$ implies that there exists a paraboloid $P$ such that items (\ref{Px0=0}) and (\ref{eq for P}) hold, and,  for all $y\in B^{\delta^{\zeta}}_\rho(x_I)$, we have
\begin{equation*}
\tu^+_\ep (y) \geq \tu^+_\ep (x)+P(y) - Ct\rho^{-1}|x-y|^3,
\end{equation*}
where $C$ is a universal constant. 
Since, according to (\ref{Bdelrhos}), we have $B_\delta(x)\subset B^{\delta^{\zeta}}_\rho(x_I)$, we have that the previous line holds for all $y\in B_\delta(x)$. In addition, for $y\in B_\delta(x)$ we have $|x-y|^3\leq \delta |x-y|^2$. We use this to bound the last term on the right-hand side of the previous line from below and obtain,
\begin{equation*}
\tu^+_\ep (y) \geq \tu^+_\ep (x)+P(y) - Ct\rho^{-1}\delta |x-y|^2 \text{ for all }y\in B_\delta(x).
\end{equation*}
\textbf{Step three.} 
Thus, to establish item (\ref{helper lem touch}) and complete the proof of the lemma, it will suffice to show that there exists a universal constant $\bar{C}$ with
\begin{equation}
\label{bound error}
Ct\rho^{-1}\delta\leq \bar{C} \delta^{1/2}.
\end{equation}
To this end, we first simplify the expression for $t$: we have,
\[
t=\left(\frac{2C_4}{C_2c_1}\right)^{1/\sigma} \rho^{\frac{n}{\sigma}-2}\delta^{-\frac{n}{\sigma}\alpha_1 -\frac{2n}{\sigma}\zeta} +t_0.
\]
We now chose $c_1$ as,
\begin{equation}
\label{chose c1}
c_1=\left(\frac{2C_4}{C_2}\right), 
\end{equation}
so that we have
\[
t=\rho^{\frac{n}{\sigma}-2}\delta^{-\frac{n}{\sigma}\alpha_1 -\frac{2n}{\sigma}\zeta} +t_0.
\]
Multiplying both sides of the previous line by $\rho^{-1}\delta$ we find,
\[
t\rho^{-1}\delta = \rho^{\frac{n}{\sigma}-3}\delta^{1-\frac{n}{\sigma}\alpha_1 -\frac{2n}{\sigma}\zeta} +t_0\delta\rho^{-1}.
\]
According to (\ref{rho and r}), we have $\rho\leq \frac{r}{2}$. We also know $r\leq 1$, so we find  $\rho^{\frac{n}{\sigma}}\leq 1$. We use this to bound the right-hand side of the previous line from above and find,
\[
t\rho^{-1}\delta \leq \rho^{-3}\delta^{1-\frac{n}{\sigma}\alpha_1 -\frac{2n}{\sigma}\zeta}+t_0\delta\rho^{-1}.
\]
By the definition of $\rho$, we have $\rho\geq \delta^\zeta$. We use this to bound the right-hand  of the previous line from above and obtain,
\[
t\rho^{-1}\delta \leq \delta^{1-\frac{n}{\sigma}(\alpha_1 + 2\zeta)-3\zeta} +t_0\delta^{1-\zeta}.
\]
Our choice of universal constants $\zeta$ and $\alpha_1$ is exactly such that each of the exponents of $\delta$ in the previous line is at least $1/2$. Therefore we find,
\[
t\rho^{-1}\delta \leq \left(1+t_0\right) \delta^{1/2}.
\]
We conclude that (\ref{bound error}) holds with $\bar{C}= C(1+t_0)$. Since both $\bar{C}$ and $t_0$ are universal constants,  the proof of this lemma is complete.
\end{proof}

\section{Proof of Theorem \ref{vague main thm}}
\label{sec pf of main thm}
Here is the precise statement of our main result:
\begin{thm}\
\label{main thm}
Assume  (U\ref{assum U}), (F\ref{ellipticity}), (F\ref{F lip}), (G\ref{f assumption}), (G\ref{g c one alpha}) and $F(0,x)\equiv 0$. Assume $u$ is a viscosity solution of (\ref{eqn for u}) and assume that $\left\{v_\delta\right\}_{\delta\geq 0}$ is a family of $\delta$-supersolutions (respectively, $\delta$-subsolutions) of (\ref{eqn for u}) with, 
\[
\etanorm{v_\delta}{U}\leq  M,
\]
 and, 
 \begin{equation}
 \label{main thm asump}
 \sup_{\bdry U} (u-v_\delta) \leq 0 \text{ (respectively, } \inf_{\bdry U} (u-v_\delta) \geq 0\text{)},
 \end{equation}
 for all $\delta$. 
There exists a constant $\bar{\delta}>0$ such that  for any $\delta\leq \bar{\delta}$,
\begin{equation}
\label{main thm conc}
\sup_U (u-v_\delta)\leq \bar{c}\delta^{\bar{\alpha}} \text{ (respectively, } \inf_U (u-v_\delta) \geq  -\bar{c}\delta^{\bar{\alpha}} \text{)}.
\end{equation}
The constant $\bar{\alpha}$ depends on $\eta$, $n$, $\lambda$ and $\Lambda$;  and $\bar{c}$ and $\bar{\delta}$ depend on  $n$, $\lambda$, $\Lambda$, $\kappa$, $M$,  $\czeroone{f}{U}$, $\conegamma{g}{\bdry U}$, $\diam U$ and the regularity of $\bdry U$.
\end{thm}

Throughout the remainder of this section, we will use $C$ and $C_i$ with $i=1,2,..$ to denote generic constants that may depend on $n$, $\lambda$, $\Lambda$, $\kappa$, $M$, $\czeroone{f}{U}$, $\conegamma{g}{\bdry U}$, $\diam U$ and the regularity of $\bdry U$. In addition, $C$ may change from line to line. 
We will give the proof of the case that $v_\delta$ is a $\delta$-supersolution; the other case is similar. 

The first step is to ``replace" $u$ by $u_\ep$ and $v_\delta$ by its inf-convolution. We formulate this as the following lemma:
\begin{lem}
\label{lem main thm}
In addition to the assumptions of Theorem \ref{main thm}, let us take $\ep<\ep_0$, where $\ep_0$ is the constant from Proposition \ref{compare u and uep}, and  let $u_\ep$ be the viscosity solution of 
\begin{equation*}
\left\{
\begin{array}{l l}
F_\ep(D^2 u_\ep, x)=f^\ep(x) &\quad \text{ in }U,\\
 u_\ep = g  &\quad \text { on }\bdry U.
 \end{array}\right.
\end{equation*}
Let us also use $v^-$ to denote the inf-convolution $(v_\delta)^{\delta^\zeta, -}$ and let $U^{\delta^\zeta}_\delta $ be the subset of $U$  given by Definition \ref{defn conv set}. We have,
\begin{equation}
\label{lem first claim}
\sup_{\bdry U^{\delta^\zeta}_\delta} (u_\ep - v^-)\leq C\delta^{\eta \zeta/2}
\end{equation}
and 
\begin{equation}
\label{lem second claim}
\sup_U(u-v_\delta) \leq \sup_{U^{\delta^\zeta}_\delta} (u^\ep -v^-)+C \delta^{\eta \zeta/2}+C\ep.
\end{equation}
\end{lem} 
\begin{proof}
According to  Proposition  \ref{c one alpha prop} and by assumption,  $u$ and $v_\delta$ are H\"older continuous with exponent $\eta$.  In addition, by the definition of $U^{\delta^\zeta}_\delta $, we have,
\[
 d(U^{\delta^\zeta}_\delta, \bdry U)\leq 2 \delta^{\zeta/2} M^{1/2} +\delta \leq C\delta^{\zeta/2}.
 \]
Hence we find,
 \[
  \sup_{\bdry U^{\delta^\zeta}_\delta} (u-v_\delta)\leq \sup_{\bdry U}(u-v_\delta) + C\delta^{\zeta \eta/2}
  \]
  and
  \begin{align*}
  \sup_U (u-v_\delta) &\leq \sup_{U\setminus U^{\delta^\zeta}_\delta}(u-v_\delta) + \sup_{U^{\delta^\zeta}_\delta} (u-v_\delta) \\
 & \leq   \sup_{\bdry U}(u-v_\delta)+  C\delta^{\zeta \eta/2}+\sup_{U^{\delta^\zeta}_\delta} (u-v_\delta) .
  \end{align*}
  According to assumption (\ref{main thm asump}), the first term on the  right-hand side of each of the previous lines is non-positive. Thus we obtain,
   \begin{equation}
   \label{thm lem eq}
  \sup_{\bdry U^{\delta^\zeta}_\delta} (u-v_\delta)\leq  C\delta^{\zeta \eta/2}
  \end{equation}
  and
  \begin{equation}
  \label{thm lem eq 2}
  \sup_U (u-v_\delta)  \leq   C\delta^{\zeta \eta/2}+\sup_{U^{\delta^\zeta}_\delta} (u-v_\delta) .
  \end{equation}
Since we chose $\ep\leq \ep_0$, we may apply Proposition \ref{compare u and uep}, and obtain, for all $x\in U$,
\begin{equation}
\label{use prop ep}
0\leq u(x)-u_\ep(x) \leq C\ep.
\end{equation}  
  In addition, according to item (\ref{item dist conv}) of Proposition \ref{prop conv}, we have, for all $x\in U$,
\begin{equation}
\label{inf cons}
0\leq v_\delta(x)-v^-(x)\leq 4 \delta^{\eta \zeta/(2-\eta)}M^{1/2}\leq C\delta^{\zeta \eta/2},
\end{equation}
where the second inequality follows since $\frac{\eta \zeta}{2-\eta}\geq \frac{\eta\zeta}{2}$. 
We use these facts to bound from above the difference of $u^\ep$ and $v^-$ that appears on the left-hand side of (\ref{lem first claim}):
\[
\sup_{\bdry U^{\delta^\zeta}_\delta} (u_\ep - v^-)\leq \sup_{\bdry U^{\delta^\zeta}_\delta}(u-v_\delta) +C\delta^{\zeta \eta/2}.
\]
We use (\ref{thm lem eq}) to bound the first term on the right-hand side of the previous line from above to obtain our first desired estimate  (\ref{lem first claim}). 
Next, we use (\ref{use prop ep}) and (\ref{inf cons}) to estimate the second term on the right-hand side of (\ref{thm lem eq 2}) from above and find,
\[
\sup_U(u-v_\delta) \leq C\delta^{\zeta \eta/2}+\sup_{U^{\delta^\zeta}_\delta}  (u_\ep -v^-)+C\ep,
\]
which is exactly the second desired estimate (\ref{lem second claim}). Thus the proof of the lemma is complete.
\end{proof}

\begin{proof}[Proof of Theorem \ref{main thm}]
Let $C_1$ and $r_0$ be the constants from  Corollary \ref{cor compare};   $\ep_0$ the constant from Proposition \ref{compare u and uep}; and $\zeta$, $\theta$ and $\alpha_1$ as given in the beginning of Section \ref{main lem sec}. We will be applying Proposition \ref{main lem} and we will be using the constants $\tilde{\delta}$, $\tilde{\alpha}$, and $\tilde{c}$ whose existence is asserted by Proposition \ref{main lem} (the constant $C_1$ that appears in the statement of Proposition \ref{main lem} will be exactly the constant $C_1$ that we have fixed here).

We define the constant $\bar{\delta}$ by
\[
\bar{\delta} = \min \left\{ \left(\frac{\ep_0}{4M^{1/2}+1}\right)^{\frac{2}{\theta}} , \tilde{\delta}, r_0^{\frac{1}{\theta}}\right\}.
\]
We take $\delta \leq \bar{\delta}$ and define the parameters $\nu$, $\ep$, and $r$ by
\[
\nu=4\linfty{v_\delta}{U}^{1/2}\delta^{\zeta/2}, \  \  \ep =\frac{\nu+\delta^\theta}{2}, \text{ and }r=\delta^\theta.
\]
Observe that since $\delta\leq \bar{\delta}$, our choices of constants imply  $\ep \leq \ep_0$ and $r\leq r_0$. In addition, we let $R$ denote $\diam U$, and  define the constants $\alpha_1$ and $\tilde{C}$ by,
\begin{equation}
\label{def alp C }
\alpha_1 = \min\left\{\theta \eta, \tilde{\alpha}, \theta\alpha\right\} \text{ and }\tilde{C} = 16R^2(\tilde{c}+2C_1).
\end{equation}

 Let $u_\ep$ and $v^-$ be as in the statement of Lemma \ref{lem main thm}.  According to item (\ref{item delta solves}) of Proposition \ref{prop conv} and  our choice of parameter $\nu$, we have that $v^-$ is a $\delta$-super solution of 
\[
F_\nu(D^2v, x)=f^\nu(x) \text{ in }U^{\delta^\zeta}_\delta.
\]
In addition, according to (\ref{lem first claim}), we have, for some constant $C_2$,
\begin{equation}
\label{use lem first claim}
\sup_{\bdry U^{\delta^\zeta}_\delta} (u_\ep - v^-)\leq C_2\delta^{\eta \zeta/2}.
\end{equation}
We use $m$ to denote 
\begin{equation}
\label{definem}
m = \sup_{ U_\delta^{\delta^{\zeta}}} (u_\ep-v^-) -C_2 \delta^{\eta \zeta/2}.
\end{equation}
We will prove, 
\begin{equation}
\label{est on m}
m\leq\tilde{C}\delta^{\alpha_1},
\end{equation}
where $\tilde{C}$ and $\alpha_1$ are given by (\ref{def alp C }). 
Once we establish (\ref{est on m}), the proof of the theorem will be complete. Indeed, according to (\ref{lem second claim}) of Lemma \ref{lem main thm}, we have 
\[
\sup_U(u-v_\delta) \leq \sup_{U^{\delta^\zeta}_\delta} (u^\ep -v^-)+C \delta^{\eta \zeta/2}+C\ep.
\]
Thus, we may use (\ref{est on m}) and the definition of $m$ to estimate the first term on the right-hand side of the previous line from above and obtain,
\[
\sup_U(u-v_\delta) \leq \tilde{C}\delta^{\alpha_1}+C_2 \delta^{\eta \zeta/2}+C \delta^{\eta \zeta/2}+C\ep.
\]
Since $\ep$ is a positive power of $\delta$, the previous line implies that the desired estimate (\ref{main thm conc}) holds.

To establish (\ref{est on m}), we proceed by  contradiction and assume 
\begin{equation}
\label{m}
m> \tilde{C} \delta^{\alpha_1}.
\end{equation}
We now apply Lemma \ref{elem lem} in $U_\delta^{\delta^{\zeta}}$ with $w(x)=u_\ep(x) -v^-(x) -C_2 \delta^{\eta \zeta/2}$. According to (\ref{use lem first claim}), $w$ is non-positive on the boundary of $U_\delta^{\delta^{\zeta}}$, and we have assumed its supremum, $m$, is positive. Hence, there exists  $x_0\in U_\delta^{\delta^{\zeta}}$ with
\begin{equation}
\label{d of x}
d(x_0, \bdry U_\delta^{\delta^{\zeta}}) \geq 2\delta^{\frac{\alpha_1}{\eta}}\geq  2\delta^{\theta}=2r,
\end{equation}
(where the second inequality and the equality follow from the definitions of $\alpha_1$, $\theta$ and $r$), and an affine function $l(x)$ such that 
\begin{equation}
\label{at x0}
l(x_0) =  u_\ep(x_0) -v^-(x_0),
\end{equation}
 and for all $x\in U_\delta^{\delta^{\zeta}}$,
\begin{equation*}
 u_\ep(x) -v^-(x)\leq l(x)-\frac{m}{2R^2}|x-x_0|^2.
\end{equation*}
In particular, we have $B_{r/2}(x_0)\subset U_\delta^{\delta^{\zeta}}$, so that we may take the supremum over $x\in B_{r/2}(x_0)$ of the previous line. We find,
\begin{equation}
\label{bdryBr/2}
\sup_{x\in \bdry B_{r/2}(x_0)} (u_\ep(x) -v^-(x) -l(x))\leq -\frac{m}{8R^2}r^2. 
\end{equation}

Next we ``freeze the coefficients" of $F_\ep$ at this point $x_0$ and define $\tu_\ep$ to be the solution of 
\begin{equation*}
\left\{
\begin{array}{l l}
F_\ep(D^2 \tu_\ep, x_0)=f^\ep(x_0) &\quad \text{ in }B_{2r}(x_0)\\
 \tu_\ep = u_\ep &\quad \text { on }\bdry B_{2r}(x_0).
 \end{array}\right.
\end{equation*}
We remark that, according to (\ref{d of x}), we have 
\begin{equation}
\label{BinU}
B_{2r}(x_0) \subset U^{\delta^\zeta}_\delta.
\end{equation}
Because $F_\ep$ and $B_{2r}(x_0)$ satisfy the assumptions of Corollary \ref{cor compare}, and  $C_1$ is exactly the constant provided by  Corollary \ref{cor compare}, we find
\begin{equation}
\label{tuep and uep and l0}
\linfty{\tu_\ep - u_\ep}{B_{2r}(x_0)}\leq C_1r^{2+\alpha}
\end{equation}
and
\begin{equation*}
\conealpha{\tu_\ep}{B_r(x_0)}^*\leq C_1.
\end{equation*}
We will be applying the first part of Proposition \ref{main lem}. The previous estimate says exactly that the hypothesis (\ref{A1}) is satisfied.  
To place ourselves exactly into the situation of Proposition \ref{main lem}, we modify $v^-$ by a affine function, and define $v$ by
\[
v(x)=v^-(x) + l(x)-\frac{m}{8R^2}r^2+C_1r^{2+\alpha}.
\]
According to item (\ref{item reg conv}) of Proposition \ref{prop conv}, $v$ satisfies hypothesis (\ref{A1.5}) of Proposition \ref{main lem}.  We will now show that $\tu_\ep-v$ is non-positive on the boundary of $B_{r/2}(x_0)$, thus verifying the remaining hypothesis (\ref{A2}). Indeed, according to (\ref{tuep and uep and l0}) and the definition of $v$, we have, for all $x\in \bdry B_{r/2}(x_0)$,
\begin{align*}
\tu_\ep(x)-v(x)&\leq  u_\ep(x)+C_1r^{2+\alpha}  -v(x)\\
&= u_\ep(x)+C_1r^{2+\alpha}    - ( v^-(x)  + l(x)-\frac{m}{8R^2}r^2+C_1r^{2+\alpha})\\
& = u_\ep(x)-v^-(x) - l(x) +\frac{m}{8R^2}r^2.
\end{align*}
According to (\ref{bdryBr/2}),  the right-hand side of the previous line is non-positive on  $\bdry B_{r/2}(x_0)$, so we find,
\[
\tu_\ep(x)-v(x)\leq 0 \text{ for all }x\in \bdry B_{r/2}(x_0).
\]
We have shown that $v$ and $\tu_\ep$ satisfy the assumptions of Part 1 of Proposition \ref{main lem} with  our choices of $\nu$ and $\ep$. Applying the proposition therefore yields,
\begin{equation}
\label{consequence of lemma}
\sup_{B_{r/2}(x_0)} \tu_\ep -v \leq \tilde{c}\delta^{\tilde{\alpha}} r^2.
\end{equation}
We will now show that (\ref{consequence of lemma}) and (\ref{at x0}) lead to a contradiction. By (\ref{tuep and uep and l0}), the definition of $v$, and (\ref{at x0}), we have
\begin{align*}
\tu_\ep(x_0)-v(x_0) &\geq u_\ep(x_0)-C_1r^{2+\alpha}  -v(x_0)\\
&= u_\ep(x_0) -C_1r^{2+\alpha}   -( v^-(x_0)  + l(x_0)-\frac{m}{8R^2}r^2+C_1r^{2+\alpha})\\
&=u_\ep(x_0) - v^-(x_0) -l(x_0) +\frac{m}{8R^2}r^2 - 2C_1r^{2+\alpha}
\end{align*}
Let us now recall that $x_0$ is exactly  the point at which the affine function $l$ touches $u_\ep-v^-$; in other words,  (\ref{at x0}) holds. Therefore, the sum of the first three terms on the right-hand side of the previous line is simply zero, and we obtain,
\begin{align*}
\tu_\ep(x_0)-v(x_0) &\geq \frac{m}{8R^2}r^2 - 2C_1r^{2+\alpha}.
\end{align*}
Rearranging yields,
\[
m \leq 8R^2(r^{-2}(\tu_\ep(x_0)-v(x_0)) +2C_1r^{\alpha}).
\]
We use the estimate  (\ref{consequence of lemma}) to bound the term in the inner-most parenthesis on the right-hand side of the previous line. Then, we recall the definitions of $\alpha_1$ and $\tilde{C}$  and find,
\[
m \leq 8R^2(\tilde{c}\delta^{\tilde{\alpha}}+2C_1\delta^{\theta\alpha})\leq \frac{\tilde{C}}{2}\delta^{\alpha_1}.
\]
But this contradicts (\ref{m}); therefore, (\ref{est on m}) must hold, and thus the proof of the theorem is complete.
\end{proof}

\section{Approximation schemes}
\label{sec h}

We now present our result on monotone finite difference approximations to (\ref{eqn for u}). First, we introduce the necessary notation and discuss our assumptions. In the next section we give the full statement of Theorem \ref{vague main thm h} and its proof. We follow the notation of \cite{Approx schemes, KT estimates, KT discrete methods}. Our mesh of discretization is
\[
E=h\zz^n =\left\{mh:\  m\in \zz^n\right\},
\]
the integer mesh of size $h$. 
We fix some $N>1$ and define the bounded subset $Y_N$ of $E$  by,
\[
Y_N=\left\{y\in E : \  0<|y|<Nh\right\}.
\]
The standard second-order difference operator $\delta_y^2$ is defined as,
\[
\delta_y^2u(x) = \frac{u(x-y)-2u(x)+u(x+y)}{|y|^2},
\]
and the collection of $\delta_y^2u(x)$ for $y$ in $Y_N$ is denoted by $\delta^2 u(x)$:
\[
\delta^2 u(x) = \left\{ \delta_y^2u(x) : \  y\in Y_N\right\}.
\]
We consider \emph{finite difference operators} $F_h$ of the form
\[
F_h[u](x) = \mathcal{F}( \delta^2 u(x) , u(x), x),
\]
where
\[
\mathcal{F}: \rr^{Y_N}\times \rr\times \rr^n \rightarrow \rr.
\]
We assume that the operators are \emph{monotone}, which means they satisfy:
\begin{enumerate}[({$F_h$}1)]
\item for all $x$ in $U$, $z, \tau\in \rr$, and $q, \eta\in \rr^{Y_N}$ such that $0\leq \eta_y\leq \tau$ for all $y\in Y_N$,
\[
F_h(q+\eta, z, x) \geq F_h(q, z, x) \geq F_h(q+\eta, z+\tau, x).
\]
\end{enumerate}
This definition of a monotone operator is equivalent to the one given in the introduction.

We say that the family of difference operators $\left\{F_h\right\}_{0\leq h\leq h_0}$ (also called a \emph{difference scheme}) is \emph{consistent with $F$} in $U$ if for each $\phi\in C^2(U)$,
\[
F_h[\phi](x) \rightarrow (F(D^2\phi,x)-f(x)) \text{ in }C(U) \text{ as }h\rightarrow 0.
\]
In \cite{KT discrete methods}, it was shown that if $F(X,x)$ is elliptic and continuous in $x$, then there exists a difference scheme $\left\{F_h\right\}$ that is consistent with $F$. 

In order to obtain an error estimate, we need to quantify the above limit. As in \cite{Approx schemes}, we make the following assumption:
\begin{enumerate}[({$F_h$}2)]
\item 
there exists a positive constant $K$ such that for all $\phi\in C^3(U)$ ,
\[
|F_h[\phi](x) - (F(D^2 \phi, x)-f(x))|\leq K(1+\linfty{D^3\phi}{U})h \text{ for all }x\in U\cap E.
\]
\end{enumerate} 
Schemes that satisfy ($F_h$2) are said to be \emph{consistent with an error estimate for $F$}. 

We divide $U\subset \rr^n$ into interior and boundary points relative to an operator $F_h$. We denote by   $U_h$ the intersection of $U$ and the mesh $E$:
\[
U_h=U\cap h\zz^n.
\]
We define the interior mesh points $U_h^i$ as,
\[
U_h^i = \left\{x\in U_h: \  d(x, \bdry U) > Nh\right\}.
\]
We observe that $F_h[u](x)$, for any $x\in U^i_h$, depends only on the values of $u$ in $U$. We define the boundary mesh points $U_h^b$  as,
\[
U_h^b=U_h\setminus U_h^i.
\]

For a mesh function $u:U_h\rightarrow \rr$ and for $V\subset U$ we define the following norms and seminorm:
\[
\linfty{u}{V} = \sup_{x\in V\cap E}|u(x)|,
\]
\[
\etasemi{u}{V} = \sup_{x,y\in V\cap E}\frac{u(x)-u(y)}{|x-y|^\eta}, 
\]
and
\[
\etanorm{u}{V} =\linfty{u}{V} +\etasemi{u}{V}.
\]

Given $g\in C^{1,\gamma}(U)$, we  consider the discrete boundary value problem
\begin{equation}\left\{
\begin{array}{l l}
\label{eqn for uh}
F_h[v_h](x)=0 &\quad \text{ in }U_h^i \\
v_h=g &\quad \text{ on }U_h^b.
\end{array}\right.
\end{equation}

It is shown in \cite{KT discrete methods, KT estimates} that (\ref{eqn for uh}) has a unique solution $v_h$ and that $v_h$ is uniformly H\"older continuous. We summarize these results:
\begin{thm}
\label{propKT}
Assume (U\ref{assum U}), ($F_h$1), ($F_h$2) and let $g\in C^{1,\gamma}(U)$.  There exists a unique solution $v_h$ of (\ref{eqn for uh}). Moreover, there exist constants  $\eta$ and $C$ that depend $n$, $\lambda$, $\Lambda$, $\conegamma{g}{U}$, $\diam U$ and the regularity of $\bdry U$ such that for every $h\in (0, 1)$,
\[
\etanorm{v_h}{U_h} \leq C.
\]
\end{thm}

\subsection{Inf and sup convolutions of mesh functions}
We recall the definitions of regularization by inf- and sup- convolution of mesh functions. This technique was used in \cite{Approx schemes}. 
\begin{defn}
For a  function $v_h$ on $U_h$ and a constant $\theta>0$, we define, for all $x\in U$,  the \emph{sup-convolution} $v_h^{\theta, +}(x)$ and the \emph{inf convolution} $v_h^{\theta, -}(x)$ by
\[
v_h^{\theta, +}(x) = \sup_{y\in U_h}\left\{v_h(y)-\frac{|x-y|^2}{2\theta}\right\} \text{ and  } v_h^{\theta, -}(x) = \inf_{y\in U_h}\left\{v_h(y)+\frac{|x-y|^2}{2\theta}\right\}.
\]
\end{defn}
\begin{defn}
\label{conv set h}
Given $h>0$, $\delta>0$, $\theta>0$ and a mesh function $v_h$ on $U_h$, we define the subset $U^{\theta}_{h,\delta}$ of $U$ by
\[
U^{\theta}_{h,\delta} = \{ x\in U \  | \  d(x, \bdry U)> 4 \theta^{1/2} \linfty{v_h}{U_h}^{1/2} +\sqrt{n}h+ \delta\}. 
\]
\end{defn}
In the appendix we summarize the basic properties of inf- and sup- convolutions of mesh functions (see Proposition \ref{prop conv h}). 

It is a classical fact of viscosity theory that if $u\in C(U)$ is the viscosity solution of $F(D^2u)=0$ in $U$, then the sup-convolution of $u$ is a subsolution of the same equation (see Proposition \ref{prop conv}). In the following proposition, we establish a similar relationship between solutions $v_h$ of (\ref{eqn for uh}) and $\delta$-solutions of (\ref{eqn for u}).

\begin{prop}
\label{v_h is delta soln} Assume that $F_h$ is a monotone scheme that is consistent with an error estimate for $F$ with constant $K$. Suppose $v_h$ is a solution of  (\ref{eqn for uh}) in $U$. Then $v_h^{\theta, +}$ is a $\delta$-subsolution of 
\[
F^\nu(D^2v, x)=f_\nu(x) -Kh \text{ in } U^{\theta}_{h,\delta},
\]
and $v_h^{\theta, -}$ is a $\delta$-supersolution of 
\[
F_\nu(D^2v, x)=f^\nu(x) +Kh \text{ in } U^{\theta}_{h,\delta},
\]
with $\delta=Nh$ and $\nu = 4\theta^{1/2}\linfty{u_h}{U_h}^{1/2}+\sqrt{n}h$.
\end{prop}

\begin{proof}
We will show that $v_h^{\theta, +}$ is a $\delta$-subsolution; the other part of the proof is very similar. Let $x\in U^{\theta}_{h,\delta}$ and let $P$ be a quadratic polynomial with
\begin{equation}
\label{P=v}
P(x)=v_h^{\theta, +}(x)
\end{equation}
and 
\begin{equation}
\label{P>v}
P(y)\geq v_h^{\theta, +}(y) \text{ for every }y\in B_{hN}(x).
\end{equation}
By definition of $\delta_y^2 $, we have, 
\begin{align*}
\delta_y^2 P(x) &= \frac{1}{|y|^2}(P(x+y)-2P(x)+P(x-y)).
\end{align*}
Let us take  $y\in Y_N$, so that $x+y\in B_{hN}(x)$.  We use (\ref{P=v}) and (\ref{P>v}) to estimate the right-hand side of the previous line from below and obtain,
\begin{align*}
\delta_y^2 P(x) &\geq \frac{1}{|y|^2}(v_h^{\theta, +}(x+y)-2v_h^{\theta, +}(x)+v_h^{\theta, +}(x-y)).
\end{align*}
(It is exactly here that it is important that $P$ stays above $v_h^{\theta, +}$ on \emph{all} of $B_{hN}(x)$.)  Let $x^*$ be a point where the supremum is achieved in the definition of $v_h^{\theta, +}(x)$. Using the definition of $v_h^{\theta, +}$ in all three terms of the previous line yields,
\begin{align*}
\delta_y^2 P(x) 
&\geq \frac{1}{|y|^2}\left(v_h(x^*+y)-\frac{|x-x^*|^2}{2\theta} - 2\left(v_h(x^*)-\frac{|x-x^*|^2}{2\theta} \right) + v_h(x^*-y)-\frac{|x-x^*|^2}{2\theta} \right)\\
&=\delta^2_y v_h(x^*),
\end{align*}
where the equality follows from the definition of $\delta^2_y $.  
Since $P$ is a quadratic polynomial, we have $\delta_y^2 P(x) = \delta_y^2 P(x^*)$. We use the monotonicity of $F_h$ and the conclusion of the previous computation to obtain
\begin{equation}
\label{FhP Fhv}
F_h[P](x^*)\geq F_h[v_h](x^*).
\end{equation}
By the properties of sup-convolutions, (see item (\ref{x and x* for vh}) of Proposition \ref{prop conv h}), we find
\begin{equation}
\label{x-x^*}
|x-x^*| \leq 4\linfty{v_h}{U_h}^{1/2}\theta^{1/2}+\sqrt{n}h .
\end{equation}
Since $x\in U^\theta_{h,\delta}$, the definitions of $U^{\theta}_{h,\delta}$ and $U^i_h$, together with the previous bound, imply 
\[
(B_{hN}(x^*)\cap E)\subset U_h^i.
\]
Because $v_h$ is a solution of (\ref{eqn for uh}) in $U_h^i$, we have $F_h[v_h](x^*)\geq 0$. Together with (\ref{FhP Fhv}), this implies 
\[
F_h[P](x^*)\geq 0.
\]
Since $F_h$ is  consistent with an error estimate for $F$, we obtain
\[
0\leq F_h[P](x^*) \leq F(D^2 P, x^*) -f(x^*) +Kh \leq F^\nu(D^2P, x)-f_\nu(x)+Kh,
\]
where the last inequality follows from (\ref{x-x^*}) and our choice of $\nu$.
\end{proof}

\section{Proof of Theorem \ref{vague main thm h}}
\label{sec proof of theorem h}

Here is the precise statement of Theorem \ref{vague main thm h}.
\begin{thm}
\label{main thm h}
Assume (U\ref{assum U}), (F\ref{ellipticity}), (F\ref{F lip}), (G\ref{f assumption}),  $F(0,x)\equiv 0$, and let us take $g\in C^{1,\gamma}(U)$. 
Assume that $\left\{F_h\right\}$ is a monotone scheme that is consistent with an error estimate for $F$ with constant $K$. Assume that $u$ is the viscosity solution of (\ref{eqn for u}) and that $v_h \in C^{0,\eta}(U)$ is the solution of (\ref{eqn for uhh}). There exist positive constants $\bar{\alpha}$, $\bar{h}$ and $\bar{c}$ such that for all $0<h\leq \bar{h}$, 
\begin{equation}
\label{claim h}
\sup_{U_h} |u-v_h|\leq \bar{c}h^{\bar{\alpha}}.
\end{equation}
The constant $\bar{\alpha}$  depends on $n$, $\lambda$, $\Lambda$ and  $\eta$; the constants $\bar{h}$ and $\bar{c}$ depend on $n$, $\lambda$, $\Lambda$, $\kappa$, $K$, $\czeroone{f}{U}$, $\conegamma{g}{ U}$,  $\diam U$ and the regularity of $\bdry U$. 
\end{thm}

The proof is very similar to that   of Theorem \ref{vague main thm}. 
Throughout the remainder of this section, we will use $C$ and $C_i$ with $i=1,2,..$ to denote generic constants that may depend on $n$, $\lambda$, $\Lambda$, $\kappa$, $K$, $\czeroone{f}{U}$, $\conegamma{g}{U}$, $\diam U$ and the regularity of $\bdry U$. In addition, $C$ may change from line to line. 

We will give the proof of the bound
\[
\sup_{U_h}u-v_h\leq \bar{c}h^{\bar{\alpha}},
\]
the proof of the other side of the estimate is similar.

The first step is to ``replace" $u$ by $u_\ep$ and $v_h$ by its inf-convolution. We formulate this as the following lemma:
\begin{lem}
\label{lem main thm h}
In addition to the assumptions of Theorem \ref{main thm h}, let us take $\ep<\ep_0$, where $\ep_0$ is the constant from Proposition \ref{compare u and uep}, and  let $u_\ep$ be the viscosity solution of 
\begin{equation*}
\left\{
\begin{array}{l l}
F_\ep(D^2 u_\ep, x)=f^\ep(x) &\quad \text{ in }U,\\
 u_\ep = g  &\quad \text { on }\bdry U.
 \end{array}\right.
\end{equation*}
We set $\delta$ to be,
\[
\delta=Nh.
\]
Let us  use $v^-$ to denote the inf-convolution $v_h^{\delta^\zeta, -}$ and let $U^{\delta^\zeta}_{h,\delta} $ be the set given by Definition \ref{conv set h}. 
We define the constant $\xi$ by,
\[
\xi =\min \left\{ \frac{\eta\zeta}{2},  \frac{\eta }{2-\eta}\right\}.
\]
We have,
\begin{equation}
\label{lem first claim h}
\sup_{\bdry U^{\delta^\zeta}_{h,\delta} } (u_\ep - v^-)\leq C\delta^{\xi}
\end{equation}
and 
\begin{equation}
\label{lem second claim h}
\sup_{U_h}(u-v_h) \leq \sup_{U^{\delta^\zeta}_{h,\delta} } (u^\ep -v^-)+C \delta^{\xi}+C\ep.
\end{equation}
\end{lem} 

\begin{proof}
The proof of this lemma is a little bit more delicate than that of Lemma \ref{lem main thm}, because the function $v_h$ is only defined on points of the mesh $U_h$, and not on the rest of $U$. Before proceeding, let us recall that, according to Proposition \ref{compare u and uep}, we have, for all $x\in U$,
\begin{equation}
\label{use prop ep h}
0\leq u(x)-u_\ep(x) \leq C\ep.
\end{equation}  
Let us also recall the definition of $U^{\delta^\zeta}_{h,\delta}$: 
\[
U^{\theta}_{h,\delta} = \{ x\in U \  | \  d(x, \bdry U)>  4 \delta^{\zeta/2} \linfty{v_h}{U_h}^{1/2} +\sqrt{n}h+ \delta\}. 
\]
We will first establish (\ref{lem first claim h}). To this end, let us fix some $x\in \bdry U^{\delta^\zeta}_{h,\delta}$ and let $y\in U_h$ be the nearest neighboring mesh point of $x$. By item (\ref{item conv holder}), we have a lower bound on $v^-(x)$ in terms of $v_h(y)$:
\begin{equation}
\label{v-x}
v^-(x) \geq v_h(y)-C\delta^{\frac{\eta }{2-\eta}}.
\end{equation}
 In addition, the definition of $U^{\delta^\zeta}_{h,\delta} $ implies that there exits a point $z$ on the discrete boundary of $U$ that is ``close" to $x$:  precisely, $z\in U^b_h$ and satisfies,
\[
|x-z|\leq  4 \delta^{\zeta/2} \linfty{v_h}{U_h}^{1/2} +\sqrt{n}h+ \delta  \leq C\delta^{\zeta/2}.
\]
Since $y$ is a neighbor of $x$, we may use the triangle inequality and the previous estimate to find that $y$ is also ``close" to $z$:
\[
|y-z|\leq |y-x|+|x-z|\leq nh+C\delta^{\zeta/2}\leq  C\delta^{\zeta/2}.
\]
According to Theorem  \ref{propKT},  $v_h$ is H\"older continuous. Therefore, we may use the previous inequality to bound $v_h(y)$ on from below in terms of $v_h(z)$:
\[
v_h(y)\geq v_h(z)-C\delta^{\frac{\eta}{2-\eta}}.
\] 
We use this to bound the first term on the right-hand side of (\ref{v-x}) from below and obtain,
\begin{equation}
\label{v-x est}
v^-(x)\geq  v_h(z) - C\delta^{\eta\zeta/2} - C\delta^{\frac{\eta }{2-\eta}} \geq v_h(z) -C\delta^\xi,
\end{equation}
where the last inequality follows from the definition of $\xi$. In addition, $u$ is H\"older continuous (according to Proposition \ref{c one alpha prop}), so we find,
\[
u(x) \leq u(z) + C\delta^{\eta\zeta/2}\leq u(z) + C\delta^\xi.
\]
Finally, we use (\ref{use prop ep h}) to estimate the left-hand side of the previous line from below by $u_\ep(x)$, and obtain
\[
u_\ep(x)\leq u(z) + C\delta^\xi.
\]
 Subtracting (\ref{v-x}) from the previous line yields:
\[
u_\ep(x)-v^-(x) \leq u(z) -v_h(z)+ C\delta^\xi.
\]
Since $z\in U^b_h$, and we have assumed that $u$ and $v_h$ agree on $U^b_h$, we have that the right-hand side of the previous line is simply $C\delta^\xi$. Moreover, this holds for all $x\in \bdry U^{\delta^\zeta}_{h,\delta}$, so we have established (\ref{lem first claim h}). 

Let us now prove that  (\ref{lem second claim h}) holds. We  use that $u$ and $v_h$ are H\"older continuous, as well as the definition of $U^{\delta^\zeta}_{h,\delta} $, to bound the left-hand side of (\ref{lem second claim h}) from above as follows:
\begin{align*}
\sup_{U_h} (u-v_h) &\leq \sup_{U_h\setminus( U^{\delta^\zeta}_{h,\delta} \cap E)} (u-v_h) +\sup_{ U^{\delta^\zeta}_{h,\delta} \cap E} (u-v_h) \\
&\leq  \sup_{U^b_h} (u-v_h) +C\delta^{\zeta\eta/2}+\sup_{ U^{\delta^\zeta}_{h,\delta} \cap E} (u-v_h).
\end{align*}
Since $u$ and $v_h$ agree on $U^b_h$, the first term is zero. Together with the definition of the constant $\xi$, this implies,
\[
\sup_{U_h} (u-v_h)\leq C\delta^\xi+\sup_{ U^{\delta^\zeta}_{h,\delta} \cap E} (u-v_h).
\]
Finally, by item (\ref{item conv sizes}), we have $v_h\geq v^-$ on all of $U_h$. We use this, together with the upper bound (\ref{use prop ep h}) on $u$ in terms of $u_\ep$, to estimate the second term on the right-hand side of the previous line from above and obtain,
\[
\sup_{U_h} (u-v_h)\leq C\delta^\xi+\sup_{ U^{\delta^\zeta}_{h,\delta} \cap E} (u_\ep-v^-)+C\ep.
\]
Since $U^{\delta^\zeta}_{h,\delta} \cap E$ is contained in $U^{\delta^\zeta}_{h,\delta}$, the proof of item (\ref{lem second claim h}), and hence of the lemma, is complete.
\end{proof}

\begin{proof}[Proof of Theorem \ref{main thm h}]
We denote $M=\sup_h\etanorm{v_h}{U}$, which is finite by Theorem \ref{propKT}. We take $C_1$ to be the constant from  Corollary \ref{cor compare}, $r_0$ the constant from Proposition \ref{compare to constant coef},  $\ep_0$ the constant from Proposition \ref{compare u and uep}, and $\zeta$ and $\theta$ as given in the beginning of Section \ref{main lem sec}.  We will be applying Proposition \ref{main lem} and we will be using the constants $\tilde{\delta}$, $\tilde{\alpha}$, and $\tilde{c}$ whose existence is asserted by Proposition \ref{main lem} (the constant $C_1$ that appears in the statement of Proposition \ref{main lem} will be exactly the constant $C_1$ that we have fixed here). 
We define
\[
\bar{h} = \frac{1}{N} \min \left\{ \left(\frac{\ep_0}{4M^{1/2}+1}\right)^{\frac{2}{\theta}} , \tilde{\delta}, r_0^{\frac{1}{\theta}}\right\}.
\]
We take $h\leq \bar{h}$ and define the parameters $\delta$, $\nu$, $\ep$ and $r$ by,
\[
\delta = Nh, \,  \nu = 4\theta^{1/2}\linfty{v_h}{U_h}^{1/2}+\sqrt{n}N\delta,\,     \ep =\nu+\delta^\theta \text{ and }r=\delta^\theta.
\]
 Since we have $h\leq \bar{h}$, our choices of the various parameters imply   $\ep \leq \ep_0$ and $r\leq r_0$.  We also denote $R=\diam U$ and set $\tilde{\alpha}$ and $\tilde{C}$ to be,
 \begin{equation}
 \label{chose alph C h}
\alpha_1 = \min\left\{\theta \eta, \tilde{\alpha}, \theta\alpha\right\} \text{ and }\tilde{C} =  16R^2\left(\tilde{c}+2C_1+ 	\frac{K}{8\lambda} \right).
 \end{equation}

Let $v^-$ and $u_\ep$ be as in Lemma \ref{lem main thm h}. 
According to Proposition \ref{v_h is delta soln} and our choice of parameter $\nu$, we have that $v^-$ is a $\delta$-supersolution of
\[
F_\nu(D^2v, x)=f^\nu(x)+Kh \text{ in } U^{\delta^\zeta}_{h,\delta}.
\]
In addition, according to Lemma \ref{lem main thm h}, we have, for some constant $C_2$,
\begin{equation}
\label{lem first claim h use}
\sup_{\bdry U^{\delta^\zeta}_{h,\delta} } (u_\ep - v^-)\leq C_2\delta^{\xi}.
\end{equation}
We use $m$ to denote,
\[
m = \sup_{ U_{\delta,h}^{\delta^{\zeta}}} (u_\ep-v^-)-C_2\delta^{\xi}.
\]
We will establish
\begin{equation}
\label{est m h}
m\leq \tilde{C}\delta^{\alpha_1},
\end{equation}
where $\tilde{C}$ and $\alpha_1$ are given by (\ref{chose alph C h}). Once we establish this estimate, the proof of the theorem will be complete. Indeed, according to (\ref{lem second claim h}) of Lemma \ref{lem main thm h}, we have,
\begin{equation*}
\sup_{U_h}(u-v_h) \leq \sup_{U^{\delta^\zeta}_{h,\delta} } (u^\ep -v^-)+C \delta^{\xi}+C\ep.
\end{equation*}
Thus, we may use (\ref{est m h}) and the definition of $m$ to estimate the first term on the right-hand side of the previous line and obtain,
\[
\sup_{U_h}(u-v_h) \leq \tilde{C}\delta^{\alpha_1} + C_2\delta^{\xi}+C \delta^{\xi}+C\ep.
\]
Since $\ep$ is a positive power of $\delta$, the desired estimate (\ref{claim h}) holds.

To establish (\ref{est m h}), we proceed by contradiction and assume 
\begin{equation}
\label{m h}
m\geq \tilde{C} \delta^{\alpha_1}.
\end{equation}
Lemma \ref{elem lem} with $U_{\delta,h}^{\delta^{\zeta}}$ instead of $U$ and $u_\ep(x) -v^-(x) - c_1\delta^{\frac{\zeta\eta}{2-\eta}} $ instead of $w(x)$ implies that there exists $x_0\in U_{\delta,h}^{\delta^{\zeta}}$ with
\begin{equation}
\label{dist x0 h}
d(x, \bdry U_{\delta,h}^{\delta^{\zeta}}) \geq 2\delta^{\frac{\alpha_1}{\eta}}\geq  2\delta^{\theta}=2r,
\end{equation}
(where the second inequality and the equality follow from our definitions  of $\alpha_1$ and of $r$) and and an affine function $l(x)$ such that 
\begin{equation}
\label{at x0 h}
l(x_0) =  u_\ep(x_0) -v^-(x_0), 
\end{equation}
and, for all $x\in U_{\delta,h}^{\delta^{\zeta}}$,
\begin{equation*}
 u_\ep(x) -v^-(x)\leq l(x)-\frac{m}{2R^2}|x-x_0|^2.
\end{equation*}
According to (\ref{dist x0 h}), we have $B_{2/r}(x_0)\subset  U_{\delta,h}^{\delta^{\zeta}}$. So, the previous inequality holds for all $x\in \bdry B_{2/r}(x_0)$. Taking the supremum over such points $x$ yields,
\begin{equation}
\label{on B2r h}
\sup_{x\in \bdry B_{2/r}(x_0)} (u_\ep(x) - v^- (x)- l(x))\leq - \frac{m}{8R^2} r^2.
\end{equation}
Next we ``freeze the coefficients" of $F_\ep$ and define $\tu_\ep$ to be the solution of 
\begin{equation*}
\left\{
\begin{array}{l l}
F_\ep(D^2 \tu_\ep, x_0)=0 &\quad \text{ in }B_{2r}(x_0)\\
 \tu_\ep = u_\ep &\quad \text { on }\bdry B_{2r}(x_0).
 \end{array}\right.
\end{equation*}
We remark that, according to (\ref{dist x0 h}), we have $B_{2r}(x_0)\subset U_{\delta,h}^{\delta^{\zeta}}$. 
Because $F_\ep$ and $B_{2r}(x_0)$ satisfy the assumptions of Proposition \ref{cor compare} and $C_1$ is the constant from Lemma \ref{lem compare}, we find
\begin{equation}
\label{Dtuep h}
\conealpha{\tu_\ep}{B_r(x_0)}^*\leq C_1
\end{equation}
and 
\begin{equation}
\label{tuep and uep and l0 h}
\linfty{\tu_\ep - u_\ep}{B_{2r}(x_0)}\leq C_1r^{2+\alpha}.
\end{equation}

We will be applying the first part of Proposition \ref{cor compare}. The estimate (\ref{Dtuep h}) says exactly that the assumption (\ref{A1}) of Proposition \ref{cor compare} is satisfied. 
To place ourselves exactly into the situation of Proposition \ref{main lem}, we perturb $v^-$ by an affine function and a small quadratic and define
\[
v(x)=v^-(x) + l(x)-\frac{m}{8R^2}r^2+C_1r^{2+\alpha} - \frac{Kh}{2\lambda}\left(|x-x_0|^2-\frac{r^2}{4}\right).
\]
By the ellipticity of $F_\nu$, we obtain
\[
F_\nu(D^2v, x) = F_\nu(D^2v^- - Kh\lambda^{-1}I, x)\leq f^\nu(x).
\]
 According to item (\ref{item reg conv h}) of Proposition \ref{prop conv h}, $v$ satisfies hypothesis (\ref{A1.5}) of Proposition \ref{main lem}.  We will now show that $\tu_\ep-v$ is non-positive on the boundary of $B_{r/2}(x_0)$, thus verifying the remaining hypothesis (\ref{A2}). To this end, we first point out, 
 \[
v(x)\geq v^-(x) + l(x)-\frac{m}{8R^2}r^2+C_1r^{2+\alpha} .
\]
 We use the previous line, together with the estimate (\ref{tuep and uep and l0 h}), to estimate $\tu_\ep-v$ from above:
 \begin{align*}
\tu_\ep(x)-v(x)&\leq u_\ep(x)+C_1r^{2+\alpha}   - ( v^-(x)  + l(x)-\frac{m}{8R^2}r^2+C_1r^{2+\alpha})\\
& = u_\ep(x)-v^-(x) - l(x) +\frac{m}{8R^2}r^2.
\end{align*}
The estimate (\ref{on B2r h}) implies that the right-hand side of the previous line is non-positive for all $x$ in $\bdry B_{2r}(x_0)$. Thus we have shown that $v$ and $\tu_\ep$ satisfy the last assumption  (\ref{A2}) of Proposition \ref{main lem}.  Therefore,
\begin{equation}
\label{bound up v}
\sup_{B_{r/2}(x_0)}( \tu_\ep -v) \leq \tilde{c}\delta^{\tilde{\alpha}} r^2.
\end{equation}
We will now show that this bound and (\ref{at x0 h}) lead to a contradiction. By (\ref{tuep and uep and l0 h}) and the definition of $v$, we have,
\begin{align*}
\tu_\ep(x_0)-v(x_0)&\geq u_\ep(x_0) - C_1 r^{2+\alpha} -v(x_0) \\
&= u_\ep(x_0) -C_1r^{2+\alpha}  -\left( v^-(x_0)  + l(x_0)-\frac{m}{8R^2}r^2+C_1r^{2+\alpha} - \frac{Kh}{2\lambda}\left(|x-x_0|^2-\frac{r^2}{4}\right)\right).
\end{align*}
Rearranging the right-hand side and using that the term in the inner-most parenthesis is at least $-\frac{r^2}{4}$ yields,
\begin{align*}
\tu_\ep(x_0)-v(x_0)&\geq  u_\ep(x_0) - v^-(x_0) -l(x_0) +\frac{m}{8R^2}r^2 - 2C_1r^{2+\alpha} - \frac{Khr^2}{8\lambda}.
\end{align*}
According to (\ref{at x0 h}), we have that the sum of the first three terms on the right-hand side of the previous line is exactly zero. Hence we find,
\begin{align*}
\tu_\ep(x_0)-v(x_0)&\geq \frac{m}{8R^2}r^2 - 2C_1r^{2+\alpha} - \frac{Khr^2}{8\lambda}.
\end{align*}
We rearrange the previous inequality to obtain an upper bound on $m$:
\[
m \leq 8R^2(r^{-2}(\tu_\ep(x_0)-v(x_0)) +2C_1r^{\alpha} + \frac{Khr^2}{8\lambda}).
\]
We use the estimate (\ref{bound up v}) to bound the term in the inner-most parenthesis from above, and find,
\[
m  \leq 8R^2(\tilde{c}\delta^{\tilde{\alpha}}+2C_1\delta^{\theta\alpha} +\frac{Khr^2}{8\lambda})\leq \frac{\tilde{C}}{2}\delta^{\alpha_1},
\]
where the second inequality follows from our choices of $\tilde{C}$ and $\alpha_1$. 
But this contradicts (\ref{m h}); therefore, (\ref{est m h}) must hold and hence the proof of the theorem is complete. 
\end{proof}

\appendix
\section{}
\label{appendix Lemma}

In this section we recall the comparison principle for viscosity solutions (\cite[Theorem 3.3]{User's guide}) and several related results.
\begin{prop}[Comparison for viscosity solutions]
\label{comparison}
Assume (F\ref{ellipticity}). If  $u, v\in C(U)$ are, respectively, a subsolution and supersolution of $F(D^2u, x)=f(x)$ in $U$ with $u\leq v$ on $\bdry U$, then $u\leq v$ in $U$.
\end{prop} 

We now provide the (quite basic) proof of Lemma \ref{barrier}. 
\begin{proof}[Proof of Lemma \ref{barrier}]
Since $c>0$, $\bar{u}$ is a subsolution of $F(D^2u, x)=f(x)$, so $\bar{u}\leq u$ on $V$ by Theorem \ref{comparison}. 

We denote $R=\diam V$, so there exists $x_0$  such that $V\subset B_R(x_0)$. For $x\in V$
we define 
\[
w(x)=\bar{u}(x) -\frac{c}{\lambda}\left(\frac{|x-x_0|^2}{2}-\frac{R^2}{2}\right).
\]
If $x\in \bdry V$, then $w(x)\geq \bar{u}(x)$. And, since $F$ is uniformly elliptic, we have
\[
F(D^2w, x)=F(D^2\bar{u} -\frac{c}{\lambda}I,x) \leq F(D^2\bar{u}, x) - c = f(x).
\]
Therefore, $w$ is a supersolution of $F(D^2u, x)=f(x)$ on $V$, so according to Theorem \ref{comparison}, we find that for all $x\in V$,
\[
u(x)\leq w(x). 
\]
We have $w(x)\leq \bar{u}(x)+\frac{cR^2}{2\lambda}$ for all $x\in V$, which, together with the previous estimate, completes the proof of the lemma.
\end{proof}

We state \cite[Theorem 3.2]{User's guide}, modified for our setting. This deep result was instrumental in establishing comparison for viscosity solutions; we use it in the proofs of Proposition \ref{compare to constant coef} and Proposition \ref{compare u and uep}.
\begin{thm}
\label{user guide thm}
Suppose that $u,v\in C(U)$ are viscosity solutions of $F(D^2u, x)=f(x)$ and $G(D^2v, x)=g(x)$ in $U$. Suppose that $(x_a, y_a)\in U\times U$ is a local maximum of 
\[
u(x) - v(y) - \frac{a}{2}|x-y|^2.
\]
Then there exist matrices $X$ and $Y$ that satisfy
\begin{equation}
\label{matrix ineq}
 -3a\left( \begin{array}{cc}
I & 0  \\
0 & I  \\
\end{array}
 \right) \leq
\left( \begin{array}{cc}
X & 0  \\
0 & -Y  \\
\end{array}
 \right)
 \leq
 3a\left( \begin{array}{cc}
I & -I  \\
-I & I  \\
\end{array}
 \right),
\end{equation}
and $F(X, x_a)=f(x_a)$, $G(Y, y_a)=g(y_a)$.
\end{thm}

Together with Theorem \ref{user guide thm}, we use the following lemma of \cite[Lemma III.1]{Ishii Lions}.
\begin{lem}
\label{matrix lemma}
There is a constant $C(n)$ such that if $(X, Y)$ are $n\times n$ matrices that satisfy (\ref{matrix ineq}) for some constant $a$,
then 
\[
||X||,||Y|| \leq C(n)\left\{a^{1/2}||X-Y||^{1/2}+||X-Y||\right\}.
\]
\end{lem}

Now we give the proof of  Lemma \ref{doubling variables property}. 
\begin{proof}[Proof of Lemma \ref{doubling variables property}]
For any $x\in  V$, we have 
\begin{align*}
\sup_{ y \in V} \left(v(x)-w(y) - \frac{a}{2}|x-y|^2\right) &= v(x) -\inf_{ y \in V} \left(w(y) + \frac{a}{2}|x-y|^2\right)\\
&\leq v(x) - w(x) + 2\linfty{Dw}{V}^2a^{-1},
\end{align*}
where the inequality follows from  the properties of inf-convolutions. Since $v=w$ on the boundary of $V$, we find
\begin{equation*}
\sup_{ y \in B_1, x\in \bdry V} \left(v(x)-w(y) - \frac{a}{2}|x-y|^2 \right)
\leq 2\linfty{Dw}{V}^2a^{-1}.
\end{equation*}
Similarly, if $y\in \bdry V$, then
\begin{equation*}
\sup_{ x \in V, y\in \bdry V} \left( v(x)-w(y) - \frac{a}{2}|x-y|^2 \right)
\leq 2\linfty{D v}{V}^2a^{-1}.
\end{equation*}
These two bounds imply the first claim of the lemma. We now proceed to give the proof of the second claim. By the definition of $(x_a,y_a)$ as a point at which the supremum is achieved, we have, for any $(x,y)\in V\times V$,
\[
u(x_a)-v(y_a)-\frac{a}{2}|x_a-y_a|^2\geq u(x)-v(y)-\frac{a}{2}|x-y|^2,
\]
so in particular, this inequality holds with $(x,y)=(x_a,y_a)$. This implies
\[
u(x_a) -\frac{a}{2}|x_a-y_a|^2\geq u(y_a),
\]
so we find
\[
\frac{a}{2}|x_a-y_a|^2\leq u(x_a)-u(y_a) \leq \linfty{Du}{V}|x_a-y_a|,
\]
from which we easily conclude $|x_a-y_a|\leq 2a^{-1}\linfty{Du}{V}$. We find $|x_a-y_a|\leq 2a^{-1}\linfty{Dv}{V}$ in a similar way.
\end{proof}

\section{}
\label{appendix infsup}
We summarize the basic properties of inf and sup convolutions that we use in this paper. We refer the reader to \cite[Proposition 5.3]{Rates Homogen}  and \cite[Lemma 5.2]{Cabre Caffarelli book} for the proof of items (\ref{item reg conv}) - (\ref{item solves}). The proof of item (\ref{item delta solves})  is very similar to that of \cite[Proposition 5.5]{Rates Homogen} and we omit it.
\begin{prop}
\label{prop conv}
Assume $u\in C(U)$.
\begin{enumerate}
\item \label{item reg conv} In the sense of distributions, $D^2u^{+,\theta}(x)\geq -\theta^{-1}I$ and $D^2u^{-,\theta}(x)\leq \theta^{-1}I$ for all $x\in U$.
\item \label{item dist conv} If $u\in C^{\eta}(U)$, for some $\eta\in (0,1]$, then for all $x\in U$,
\begin{align*}
0\leq (u^{+,\theta}-u)(x)&\leq \etasemi{u}{U}(2\theta)^{\frac{\eta}{2-\eta}} , \text{ and}\\
0\leq (u-u^{-,\theta})(x)&\leq \etasemi{u}{U}(2\theta)^{\frac{\eta}{2-\eta}} .
\end{align*}
\item 
\label{item solves}
Define $\nu = 4\theta^{1/2}\linfty{u}{U}^{1/2}$.
If  $u$ is a subsolution of  (\ref{eqn for u}) in $U$, then $u^{\theta, +}$ is a subsolution of 
\[
F^\nu(D^2u, x)=f_\nu(x) \text{ in } U^{\theta}_\delta;
\]
if $u$ is a supersolution of  (\ref{eqn for u}) in $U$, then $u^{\theta, -}$ is a supersolution of 
\[
F_\nu(D^2u, x)=f^\nu(x) \text{ in } U^{\theta}_\delta.
\]
(The perturbed nonlinearities $F^\nu$ and $F_\nu$, as well as $f_\nu$ and $f^\nu$, are defined in Definition \ref{def ep}.)
\item \label{item delta solves} Assume that $v\in C(U)$. Let $\nu = 4\theta^{1/2}\linfty{v}{U}^{1/2}$.
If  $v$ is a $\delta$-subsolution of  (\ref{eqn for u}) in $U$, then $v^{\theta, +}$ is a $\delta$-subsolution of 
\[
F^\nu(D^2v, x)=f_\nu(x) \text{ in } U^{\theta}_\delta;
\]
if $v$ is a $\delta$-supersolution of  (\ref{eqn for u}) in $U$, then $v^{\theta, -}$ is a $\delta$-supersolution of 
\[
F_\nu(D^2v, x)=f^\nu(x) \text{ in } U^{\theta}_\delta.
\]
\end{enumerate}
\end{prop}

\subsection{Inf and sup convolutions of mesh functions}
We summarize some basic properties of inf and sup convolutions of mesh functions.
\begin{prop}
\label{prop conv h}
Assume $v\in C^{0,\eta}(U_h)$.
\begin{enumerate}
\item
\label{x and x* for vh} If $x^*\in U_h$ denotes a point where the supremum (resp. infimum) is achieved in the definition of $v_h^{\theta, +}(x)$ (resp. $v_h^{\theta, -}(x)$), then
\[
|x-x^*|\leq 4\linfty{v_h}{U_h}^{1/2}\theta^{1/2}+\sqrt{n}h.
\]
\item \label{item reg conv h} In the sense of distributions, $D^2v^{+,\theta}(x)\geq -\theta^{-1}I$ and $D^2v^{-,\theta}(x)\leq \theta^{-1}I$ for all $x\in U$. 
\item \label{item conv sizes} For  all $x\in U_h$, we have  $v^{-,\theta}(x)\leq v(x)\leq v^{+,\theta}(x)$.
\item \label{item conv holder} There exists a constant $C$ that depends on $\etasemi{v}{U}$ such that if $x\in U$ and $y$ is a neighboring mesh point to $x$,  then,
\begin{align*}
&v^{+,\theta}(y)-C\theta^{\frac{\eta}{2-\eta}} \leq v(x)\leq v^{-,\theta}(y)+C\theta^{\frac{\eta}{2-\eta}} .
\end{align*}
\end{enumerate}
\end{prop}

\begin{proof}[Proof of (\ref{x and x* for vh}) of Proposition \ref{prop conv h}]
For $x\in U$, we denote by $x_h$ an element of the mesh that is closest to $x$. Note $|x-x_h|\leq \sqrt{n}h$. 

Let $x^*\in U_h$ be a point where the supremum is achieved in the definition of $v_h^{\theta, +}(x)$. Then
\[
v_h(x^*)-\frac{|x-x^*|^2}{2\theta}\geq v_h(x_h) - \frac{|x_h-x|^2}{2\theta}.
\]
Therefore,
\[
\frac{|x-x^*|^2}{2\theta} \leq v_h(x^*) - v_h(x_h) +\frac{nh^2}{2\theta} \leq 2\linfty{v_h}{U_h} +\frac{nh^2}{2\theta}, 
\]
which easily implies the desired bound. The proof for $v_h^{\theta, -}(x)$ is very similar.
\end{proof}
We refer the reader to \cite[Proposition 2.3]{{Approx schemes}} for the proof of the rest of Proposition \ref{prop conv h}.

\section*{Acknowledgements} The author thanks her thesis advisor, Professor Takis Souganidis, for suggesting this problem and for his guidance and encouragement.


\begin{thebibliography}{9}


\bibitem{BarlesJakobsen2002}
Barles, G.; Jakobsen, E. R. 
On the convergence rate of approximation schemes for Hamilton- Jacobi-Bellman equations. 
M2AN Math. Model. Numer. Anal. 36 (2002), no. 1, 33-54.

\bibitem{BarlesJakobsen2005}
Barles, G.; Jakobsen, E. R. 
Error bounds for monotone approximation schemes for Hamilton- Jacobi-Bellman equations. 
SIAM J. Numer. Anal. 43 (2005), no. 2, 540-558

\bibitem{BarlesSouganidis}
Barles, G.; Souganidis, P. E.
Convergence of approximation schemes for fully nonlinear second order equations. 
Asymptotic Anal. 4 (1991), no. 3, 271-283. 

\bibitem{Bonnans}
Bonnans, J. F.; Maroso, S.; Zidani, H. 
Error estimates for stochastic differential games: the adverse stopping case. 
IMA J. Numer. Anal. 26 (2006), no. 1, 188-212.

\bibitem{Cabre Caffarelli book}
Caffarelli, Luis A.; Cabre, Xavier. 
Fully nonlinear elliptic equations. 
American Mathematical Society Colloquium Publications, 43. 
American Mathematical Society, Providence, RI, 1995.



\bibitem{Approx schemes}
Caffarelli, L. A.; Souganidis, P. E. 
A rate of convergence for monotone finite difference approximations to fully nonlinear, uniformly elliptic PDEs. 
Comm. Pure Appl. Math. 
61 (2008), no. 1, 1-17. 

\bibitem{Rates Homogen}
Caffarelli, L.; Souganidis, P. E. 
Rates of convergence for the homogenization of fully nonlinear uniformly elliptic pde in random media. 
Invent. Math. 180 (2010).
no. 2, 301-360.

\bibitem{User's guide}
Crandall, Michael G.; Ishii, Hitoshi; Lions, Pierre-Louis.
User's guide to viscosity solutions of second order partial differential equations. 
Bull. Amer. Math. Soc. (N.S.) 27 (1992).
no. 1, 1-67. 

\bibitem{Ishii Lions}
Ishii, H.; Lions, P.-L. 
Viscosity solutions of fully nonlinear second-order elliptic partial differential equations. 
J. Differential Equations 83 (1990).
no. 1, 26-78. 

\bibitem{Jakobsen2004}
Jakobsen, E. R. 
On error bounds for approximation schemes for non-convex degenerate elliptic equations. 
BIT 44 (2004), no. 2, 269-285.

\bibitem{Jakobsen2006}
Jakobsen, E. R. 
On error bounds for monotone approximation schemes for multi-dimensional Isaacs equations. 
Asymptot. Anal. 49 (2006), no. 3-4, 249-273.

\bibitem{Krylov 1998}
Krylov, N. V.
On the rate of convergence of finite-difference approximations for Bellman's equations. Algebra i Analiz 9 (1997), no. 3, 245--256
translation in St. Petersburg Math. J. 9 (1998), no. 3, 639-650 

\bibitem{Krylov 1999}
Krylov, N. V.
On the rate of convergence of finite-difference approximations for Bellman's equations with variable coefficients. 
Probab. Theory Related Fields 117 (2000), no. 1, 1-16. 


\bibitem{Krylov2005}
Krylov, Nicolai V.
The rate of convergence of finite-difference approximations for Bellman equations with Lipschitz coefficients. 
Appl. Math. Optim. 52 (2005), no. 3, 365-399. 


\bibitem{KT estimates}
Kuo, Hung Ju; Trudinger, Neil S.
Linear elliptic difference inequalities with random coefficients. 
Math. Comp. 55 (1990), no. 191, 37-53. 

 \bibitem{KT discrete methods}
 Kuo, Hung Ju; Trudinger, Neil S.
Discrete methods for fully nonlinear elliptic equations.
SIAM J. Numer. Anal. 29 (1992), no. 1, 123-135.

\bibitem{Winter}
Winter, Niki.
$W^{2,p}$ and $W^{1,p}$-estimates at the boundary for solutions of fully nonlinear, uniformly elliptic equations. 
Z. Anal. Anwend. 28 (2009), no. 2, 129-164. 

\end{thebibliography}
\end{document}